%% file: main.tex
\documentclass{article}
\usepackage{graphicx} 
\usepackage{mathrsfs}
\usepackage{geometry}
\usepackage{titletoc}
\usepackage{amsmath}
\allowdisplaybreaks[4]
\usepackage{amssymb}
\usepackage{mathtools}
\usepackage{enumerate}
\usepackage{amsthm}
\usepackage{caption}
\usepackage{subcaption}
\usepackage{makecell}
\usepackage{rotating}
\usepackage{mleftright}
\mleftright
\usepackage[backref,colorlinks,linkcolor=blue]{hyperref} 

\usepackage[table]{xcolor} 
\usepackage[all]{xy}
\usepackage{tikz}

\newcommand\spec{\mathrm{Spec}}
\newcommand\proj{\mathrm{Proj}}

\newcommand\sse{\subseteq}

\newcommand\tx{\widetilde{X}}

\newcommand\geqs{\geqslant}
\newcommand\mbz{\mathbb{Z}}
\newcommand\mbc{\mathbb{C}}

\newcommand\clx{\mathrm{Cl}(X)}
\newcommand\aut{\mathrm{Aut}}

\newcommand\HilbA{\Hilb_k^{K_X}}
\newcommand{\RTZ}{RTZ balanced}

\newcommand\cB{{\mathcal B}}

\newcommand\cH{{\mathcal H}}

\newcommand\cO{{\mathcal O}}

\newcommand{\id}{{\rm id}}

\newcommand{\FS}{\mathrm{FS}}
\newcommand{\Hilb}{\mathrm{Hilb}}

\newcommand{\Vol}{\mathrm{Vol}}

\newcommand{\Bc}{\mathrm{Bc}}

\newcommand{\lct}{\mathrm{lct}}
\newcommand{\Hom}{\mathrm{Hom}}

\newcommand{\CC}{\mathbb {C}}

\newcommand{\NN}{{\mathbb N}}
\newcommand{\PP}{{\mathbb P}}
\newcommand{\QQ}{{\mathbb Q}}
\newcommand{\RR}{{\mathbb R}}
\newcommand{\ZZ}{{\mathbb Z}}

\DeclareMathOperator{\Ric}{Ric}

\DeclareMathOperator{\Aut}{Aut}
\DeclareMathOperator{\Ver}{Vert}

\DeclareMathOperator{\co}{co}

\DeclareMathOperator{\diag}{diag}
\DeclareMathOperator{\tr}{tr}

\theoremstyle{plain}
\newtheorem{theorem}{Theorem}[section]
\newtheorem{proposition}[theorem]{Proposition}

\newtheorem{lemma}[theorem]{Lemma}

\newtheorem{claim}[theorem]{Claim}
\newtheorem{corollary}[theorem]{Corollary}

\newtheorem{conjecture}[theorem]{Conjecture}

\newtheorem{problem}[theorem]{Problem}

\theoremstyle{definition}
\newtheorem{definition}[theorem]{Definition}
\newtheorem{remark}[theorem]{Remark}
\newtheorem{example}[theorem]{Example}

\def\KE{K\"ahler--Einstein }
\newcommand{\beq}{\begin{equation}}
\newcommand{\eeq}{\end{equation}}
\newcommand{\bpf}{\begin{proof}}
\newcommand{\epf}{\end{proof}}
\newcommand{\bdefn}{\begin{definition}}
\newcommand{\edefn}{\end{definition}}
\newcommand{\bremark}{\begin{remark}}
\newcommand{\eremark}{\end{remark}}
\newcommand{\bconj}{\begin{conjecture}}
\newcommand{\econj}{\end{conjecture}}
\newcommand{\bcor}{\begin{corollary}}
\newcommand{\ecor}{\end{corollary}}
\newcommand{\blem}{\begin{lemma}}
\newcommand{\elem}{\end{lemma}}
\newcommand{\bclaim}{\begin{claim}}
\newcommand{\eclaim}{\end{claim}}
\newcommand{\bprob}{\begin{problem}}
\newcommand{\eprob}{\end{problem}}
\newcommand{\bprop}{\begin{proposition}}
\newcommand{\eprop}{\end{proposition}}
\newcommand{\bthm}{\begin{theorem}}
\newcommand{\ethm}{\end{theorem}}
\newcommand{\bexam}{\begin{example}}
\newcommand{\eexam}{\end{example}}
\def\lb#1{\label{#1}}
\def\ra{\rightarrow}
\def\ran{\rangle}
\def\lan{\langle}
\def\q{\quad}

\newcommand\blfootnote[1]{%
	\begingroup
	\renewcommand\thefootnote{}\footnote{#1}%
	\addtocounter{footnote}{-1}%
	\endgroup
}

\title{Symmetry notions for toric Fanos}
\author{Chenzi Jin, Yanir A. Rubinstein, Yang Zhang }
\date{20 June 2025}

\begin{document}

\maketitle

\begin{abstract}
We survey various notions of symmetry for toric varieties. These
notions range from algebraic geometric,
complex geometric, representation theoretic, combinatorial, convex geometric, to geometric stability. The main theorem gives the relationship between these notions. While mostly folklore knowledge, this does not seem to be readily available in the literature. Finally, we take the opportunity to give an accessible and simplified proof of Demazure's 1970 structure theorem for the automorphism group of a smooth toric variety, previously considered quite inaccessible.

\end{abstract}

\blfootnote{
The research of Y.A.R. was supported by the
Bergman Distinguished Visiting Professorship at Stanford University, a Simons Fellowship in Mathematics, BSF grant 2020329, and
NSF grant DMS-2506872.
The authors 
thank I. Arzhantsev, D. Cox,
S. Saito, and C. Shramov for comments.
}


\tableofcontents

\section{Introduction}

Symmetry is a manifestation of beauty in nature. 
Toric varieties are among the most symmetric complex varieties, and this
leads to rich and fascinating structures on them. In fact, their rich symmetry 
is the main
reason that toric varieties lie at the intersection of so many branches of mathematics: algebra, analysis, geometry, combinatorics, probability, dynamics, spectral theory, mirror symmetry, and the list goes on. Many mathematicians have been fascinated by the beauty
and depth of toric varieties, e.g., Demazure, Fulton, Mumford, Atiyah, Sternberg, Mabuchi, Oda, Stanley, Guillemin, Cox, Calabi, Donaldson, Zelditch. 
And, among toric varieties, {\it toric Fano} varieties are arguably the
most symmetric.

Symmetry for a differential geometer might mean the existence of a canonical geometric structure that makes the manifold most aesthetically pleasing.
For a combinatorialist it might mean a certain configuration of lattice points is arranged neatly. For a representation theorist these lattice points could simply represent characters. And for an algebraic geometer this might manifest in a certain group of automorphisms being linear algebraic and reductive, or by a certain stability or log canonical threshold being as large as possible. 

The purpose of this mostly expository article is to survey these different notions of symmetry and organize the relationships between them. The world of toric Fanos magically captures all of them, and is therefore a wonderful interdisciplinary mini-cosmos.
What is more, many general (i.e., non-toric) conjectures and problems gain tremendously from a profound understanding of the toric setting, see, e.g.,
\cite{JRT1}. Thus, one should not underestimate or belittle toric algebra/geometry.
 
\bigskip
We also take the opportunity to present a self-contained proof of Demazure's structure theorem for the automorphism group of a toric variety.
One of the reasons for doing that is that this is a beautiful theorem and when we asked around experts it seemed that none (except David Cox) were familiar with the proof---it has become a mostly forgotten one. Even Oda's famous treatise on toric varieties omits the proof, in fact, Oda writes \cite[p. 140]{Oda88},
{\sl``Demazure states his results in this final form in \cite[Proposition 11, p. 581]{Dem70} although the proof, which is valid for (not necessarily compact) non-singular toric varieties, is spread over the entire paper. We omit the proof here, since it requires so much space and basic knowledge on linear algebraic groups.''} In fact, Demazure's paper is over 80 pages long.
The purpose of \S\ref{DemazureTheo} and \S\ref{Proof of Demazure's Theorem} is to both resurrect this theorem and to simplify its proof (see Remark \ref{simplifyremark}). 
Our main resource is Cox's revisiting of Demazure's theorem to simplicial complete toric varieties from 1995 \cite{Cox95}. Cox's treatment, while more concise than Demazure's, is still quite technical and accessible only to experts.
Additionally, it originally contained an oversight, fixed by Cox in his erratum in 2014 \cite{Cox14}, once
again emphasizing the subtle nature of the theorem.

\subsection{Fundamental theorem on symmetry notions on toric Fanos}
\label{Symmetry notions on toric Fanos}

\begin{definition}
\label{SymmDef}
Let $X$ be a toric Fano manifold, $P\subseteq M$ the polytope associated to $-K_X$, and $\Aut P\subseteq \Aut M\cong \mathrm{GL}(n, \ZZ)$ the subgroup of the automorphism group of the lattice $M$ that leaves $P$ invariant. We say



\begin{enumerate}[(a)]
    \item $P$ is \emph{centrally symmetric} if $P=-P$, i.e., $x\in P$ if and only if $-x\in P$.
    \item $P$ is \emph{Batyrev--Selivanova symmetric} if the only element of $M$ fixed by $\Aut P$ is $0$. We say $X$ is \emph{Batyrev--Selivanova symmetric}  if the associated polytope $P$ is.
    \item $P$ is \emph{centrally lattice symmetric} if $R(P)=-R(P)$, where $R(P)\subset\partial P\cap M$ is the set of lattice points contained in relative interiors of the codimension-1 faces of $P$ (see Definition \ref{roots} and Figure \ref{Fig for R(P)}).
\end{enumerate}
\end{definition}

The following summarizes the different notions of symmetry on toric Fano manifolds.
It relates symmetry notions
(group actions), combinatorial 
notions (lattice point barycenters),
 differential geometric notions (canonical metrics: balanced metrics and \KE metrics),
 and an algebraic notion (reductivity).
 
\begin{theorem}\label{relation of symm notions on metric level}
Let $X$ be a toric Fano manifold and $P$ its associated
lattice polytope. 
The relationships among the following notions are: 
\hfill\break
$$
\hbox{\centerline{\noindent\textup{(\ref{item central sym-intro})}
$\Rightarrow$
\textup{(\ref{item BS sym-intro})}
$\Leftrightarrow$
\textup{(\ref{item alphaGone-intro})}
$\Rightarrow$
\textup{(\ref{item Bc_k=0 for all k-intro})}
$\Leftrightarrow$
\textup{(\ref{item Bc_k=0 for large k-intro})}
$\Leftrightarrow$
\textup{(\ref{item Bc_k=0 for n+1 values of k-intro})}
$\Rightarrow$
\textup{(\ref{item Bc=0-intro})}
$\Rightarrow$
\textup{(\ref{item R(P)=-R(P)-intro})}
$\Leftrightarrow$
\textup{(\ref{item reduc-intro})}
}}.$$

Moreover, these implications are optimal, i.e.:
$$
\eqref{item central sym-intro}\not\Leftarrow\eqref{item BS sym-intro}\not\Leftarrow\eqref{item Bc_k=0 for all k-intro}\not\Leftarrow\eqref{item Bc=0-intro}\not\Leftarrow\eqref{item R(P)=-R(P)-intro}.
$$

\noindent
In dimension at most 6, 
\textup{(\ref{item central sym-intro})}
$\Rightarrow$
\textup{(\ref{item BS sym-intro})}
$\Leftrightarrow$
\textup{(\ref{item alphaGone-intro})}
$\Leftrightarrow$
\textup{(\ref{item Bc_k=0 for all k-intro})}
$\Leftrightarrow$
\textup{(\ref{item Bc_k=0 for large k-intro})}
$\Leftrightarrow$
\textup{(\ref{item Bc_k=0 for n+1 values of k-intro})}
$\Leftrightarrow$
\textup{(\ref{item Bc=0-intro})}
$\Rightarrow$
\textup{(\ref{item R(P)=-R(P)-intro})}
$\Leftrightarrow$
\textup{(\ref{item reduc-intro})}.

\begin{enumerate}[$(\rm i)$]
\item \label{item central sym-intro}
$P$ is centrally symmetric, i.e., $P=-P$.
\item \label{item BS sym-intro}
$P$ is Batyrev--Selivanova symmetric ({Definition
\ref{SymmDef}} (b)).
\item \label{item alphaGone-intro}
$\alpha_G(-K_X)=1$, where $G\subset\Aut X$ is the group generated by $(S^1)^n$ and $N(T)/T$, $N(T)$ is the normalizer of the maximal (complex) torus $T$ in $\Aut X$;
\item \label{item Bc_k=0 for all k-intro}
$\Bc_k(P)=0$ for all $k$,
equivalently 
$\delta_k(-K_X)=1$ for all $k$, 
equivalently $X$ admits $k$-anticanonically balanced metrics for all $k$.
\item  \label{item Bc_k=0 for large k-intro}
$\Bc_k(P)=0$ for all sufficiently large $k$, 
equivalently 
$\delta_k(-K_X)=1$ for all sufficiently large $k$.
equivalently $X$ admits $k$-anticanonically balanced metrics for all sufficiently large $k$;
\item  \label{item Bc_k=0 for n+1 values of k-intro}
$\Bc_k(P)=0$ for at least $n+1$ values of $k$.
equivalently 
$\delta_k(-K_X)=1$ for at least $n+1$ values of $k$, 
equivalently $X$ admits $k$-anticanonically balanced metrics for at least $n+1$ values of $k$.
\item  \label{item Bc=0-intro}
$\Bc(P)=0$, equivalently 
$\delta(-K_X)=1$, equivalently $X$ admits a K\"ahler--Eintein metric.
\item \label{item R(P)=-R(P)-intro}
$P$ is {centrally lattice symmetric} (Definition \ref{SymmDef} (c)).
\item \label{item reduc-intro}
$\Aut_0X$ is reductive.
\end{enumerate}
\end{theorem}

For the definitions and a discussion on the canonical metrics mentioned in Theorem \ref{relation of symm notions on metric level}, namely K\"ahler--Einstein metrics (Definition \ref{KE def}) and anticanonically balanced metrics (Definition \ref{abm def}), see \S\ref{Notions of canonical metrics}.

As expected, this theorem generalizes to {\it toric Fano pairs},
i.e., to polarized $(X,L)$ with $-K_X-L>0$, where, e.g., twisted
\KE and anticanonically balanced metrics come into play \cite[\S2]{RTZ21}. We leave this and similar
generalizations to the interested readers.

It is natural to wonder whether 
the (somewhat surprising) implication 
(\ref{item Bc_k=0 for n+1 values of k-intro})
$\Rightarrow$
\textup{(\ref{item Bc=0-intro})}
is optimal. What if one further weakens the assumption and only requires the
vanishing of $\Bc_k(P)=0$
for a {\it single} value of $k$?
Specifically, if
$\Bc_k(P)=0$ for {\it some} $k$,
does that imply \textup{(\ref{item Bc=0-intro})}, i.e., $\Bc(P)=0$? 
If we allow $X$ to be singular, a counterexample is given by Nill \cite[\S5.6]{Nil04}. However, it seems no counterexample is known in the smooth case.
In a forthcoming paper, it is shown that
$\Bc_k(P)=0$ for  some $k$ implies
\textup{(\ref{item R(P)=-R(P)-intro})}
(i.e., by Theorem \ref{relation of symm notions on metric level}, also 
\textup{(\ref{item reduc-intro})}, i.e.,
$\Aut_0X$ is reductive) \cite{BJRS}.

\subsection{The Demazure structure theorem re-revisited}

\input{introduction}



\section{Preliminaries}

\subsection{Fans, cones, charts, and associated toric varieties}

Consider a lattice of rank $n$ and its dual lattice 
$$
N, \quad
M:=N^*:=
\mathrm{Hom}(N,\ZZ).
$$
Denote the corresponding $\RR$-vector space and its dual (both isomorphic to $\RR^n$)
$$
N_\RR:=
N\otimes_\ZZ\RR,
\quad
M_\RR:=M\otimes_\ZZ\RR=N_\RR^*. 
$$
A rational convex polyhedral cone in $N_\RR$ takes the form
$$
\sigma=\sigma(v_1,\ldots,v_d):=
\Big\{
{\textstyle\sum}_{i=1}^da_iv_i\,:\, a_i\ge0, v_i\in N
\Big\}.
$$
The rays $\RR_{+}v_i, i\in\{1,\ldots,d\}$ are called the generators
of the cone \cite[p. 9]{Ful93}. 
They are (1-dimensional) cones themselves, of course.
Our convention will be that the 
$$
v_i, \quad i\in\{1,\ldots,d\}, \q
\hbox{are primitive elements 
of the lattice $N$},
$$
which means there is no $m\in \NN\setminus\{1\}$ such that 
$v_i/m\in N$.
A cone is called strongly convex if
$\sigma\cap -\sigma =\{0\}$ \cite[p. 14]{Ful93}.
A face of $\sigma$ is any intersection of $\sigma$
with a supporting hyperplane.

\bdefn
\label{FanDefn}
A {\it fan} $\Delta=\{\sigma_i\}_{i=1}^\delta$ in $N$ is a finite set of rational strongly convex polyhedral
cones $\sigma_i$ in $N_\RR$ such that: 

\smallskip
\noindent 
(i) each face of a cone in $\Delta$ is also (a cone) in $\Delta$,

\smallskip
\noindent
(ii) the intersection of two cones in $\Delta$ is a face of each. 
\edefn

Such a fan gives rise to a toric variety $X_\Delta$: 
Each cone $\sigma_i$ in $\Delta$ has a dual cone in $M_\RR$
\[\sigma^*_i=\{m\in M_\RR:\, \lan m, v\ran\geqs0, \forall v\in\sigma_i \} \]
which gives rise to an affine toric variety $\spec\mbc[\sigma^*_i\cap M]$ \cite[\S1.3]{Ful93}. This serves as a (Zariski open) chart in $X_\Delta$ with the transition
between the charts constructed by (i) and (ii) above \cite[p. 21]{Ful93}.
For instance, the zero cone corresponds to the open dense orbit
$T_X\cong(\CC^*)^n$ \cite[p. 64]{Cannas}, and more generally there is a
bijection between the cones $\{\sigma_i\}_{i=1}^\delta$ and the orbits
of the torus action in $X_\Delta$
\cite[Proposition 5.6.2]{Cannas}, with 
the non-zero cones corresponding precisely to all the toric subvarieties of $X_\Delta$ of positive codimension. 
$T_X$ acts naturally on itself, and this toric action extends to $X$. We can view $M$ as the character group of $T_X$. Namely
$$
M=\{m:T_X\to\mbc^*:\, m \hbox{ is a morphism and a group homomorphism}. \}
$$

\subsection{Regular, simplicial, and complete fans}

Here are some basic terminology about cones and fans:

\begin{definition}\label{fanTermi}
    A cone is 
    \begin{itemize}
        \item[$\bullet$] smooth (or regular) if it is generated by a subset of a basis of $N$.
    \item[$\bullet$] simplicial if it is generated by a subset of a basis of $N_\RR$.
    \end{itemize}

    A fan is 
    \begin{itemize}
        \item[$\bullet$]  smooth (or regular) if every cone in it is smooth (regular).
    \item[$\bullet$] simplicial if every cone in it is simplicial.
    \item[$\bullet$] complete if the support of it (i.e., the union of all the cones in it) is the whole $N_\RR$.
    \end{itemize}

\end{definition}

Then we have the following results:

\begin{theorem}[{\cite[Theorem 3.1.19]{CLS11}}]\label{Property}
    Let $X$ be the toric variety determined by a fan $\Delta$, then:
    \begin{enumerate}[$(\rm i)$]
        \item $X$ is compact (under the complex analytic topology, or equivalently proper under the Zariski topology) if and only if $\Delta$ is complete. 
        \item $X$ is smooth (regular) if and only if $\Delta$ is smooth (regular).
        \item $X$ is an orbifold if and only if $\Delta$ is simplicial.
        \end{enumerate}
\end{theorem}

Thus, we say that a toric variety is complete or simplicial if the corresponding fan is complete or simplicial, respectively.

\begin{example}
    The weighted projective space $\mathbb{P}(1, 1, 2)$ is a complete, simplicial but non-smooth toric variety.

    For a cone $\sigma$, the corresponding affine toric variety $X_\sigma$ is an incomplete toric variety.
\end{example}

{\bfseries In this paper, we always assume the toric variety and the fan is complete.}

\subsection{Polytopes associated to fans}

Let $\Delta$ be a complete fan and let $\sigma_1,\ldots,\sigma_d$ be its 1-dimensional cones
(i.e., rays) generated by primitive generators 
\beq
\lb{Delta1Eq}
\Delta_1:=
\{v_1,\ldots,v_d\}\subset N,
\eeq
so $\sigma_i=\RR_{+}v_i$, and set
\beq
\lb{PDef2Eq}
P=\bigcap_{i=1}^d\Big\{y\in M_\RR\,:\, \langle y, -v_i\rangle \le 1\Big\}
=\{h_{-\Delta_1}\le 1\}=\{y\in M_\RR\,:\, \max_j\langle -v_j,y\rangle\le1\}.
\eeq

Note that $P$ contains the origin in its interior.
Also note that \eqref{PDef2Eq} is the standard convention since
then lattice points of $P$ correspond to monomials%
. Batyrev--Selivanova use $-P$ instead.




Then we can define the \emph{roots} of a toric variety $X$ and the associated polytope $P$ (see Figure \ref{Fig for R(P)}). 

\begin{definition}\label{roots}
    The \emph{roots} of a toric variety $X$ and the associated polytope $P$ are the set
    $$
    R(P):=\{m\in M\,:\,\exists v\in\Delta_1 {\ with\ } \langle m,v\rangle=-1 {\ and\ } \langle m,v'\rangle\geqslant 0\; \forall v'\neq v\in\Delta_1 \}.
    $$
    It can be divided into two parts:
    $$
    R_s(P):=-R(P)\cap R(P),\;\;\;R_u:=R(P)\setminus R_s(P).$$
    We call the elements in $R_s(P)$ the \emph{semisimple roots} of $X$ and the associated polytope $P$.
\end{definition}

$R(P)$ are exactly the set of lattice points contained in relative interiors of the codimension-1 faces of $P$ since: $\langle m,v\rangle\geqs -1$ for all $v\in \Delta_1$ means $m$ is in the boundary of $P$, and $\langle m,v_i\rangle=-1$ means that $m$ is in the codimension-1 faces of $P$ vertical to $v_i$, hence by definition any $m\in R(P)$ is contained in exactly one codimension-1 face of $P$, thus the relative interior of this face.

\begin{remark}
    In general, the polytope defined above may not have good properties. The vertices of $P$ may even not be the lattice points in $M$. But when $X$ is a Gorenstein variety, which means that the canonical divisor of $X$ is Cartier, $P$ is always a lattice polytope (i.e., every vertex of $P$ is a lattice point).  See \cite[Theorem 8.3.4]{CLS11}. Especially, when $X$ is smooth, $P$ is always a lattice polytope.

When $X$ is Fano, the codimension-$1$ facets of $P$ $$
\{y\in P\,:\,\langle y, -v_i\rangle = 1\}
$$
are in one-to-one correspondence with the $1$-dimensional cones generated by $v_i$ in $\Delta$, thus the irreducible toric divisors in $X$.
\end{remark}

\subsection{Automorphisms of toric varieties}

For a compact toric variety $X$, we denote the automorphism group of $X$ by $\Aut X$, and the connected component of $X$ that contains identity by $\Aut_0 X$.

As standard, we allow the slightly more flexible situation
of any (and not just a maximal) compact subgroup of the normalizer
$$N((\CC^*)^n)$$
of the complex torus $(\CC^*)^n$ in $\Aut X$.
Denote 
by 
\beq
\lb{AutPEq}
\Aut P\subseteq \mathrm{GL}(M)\cong \mathrm{GL}(n,\ZZ)
\eeq
the subgroup of the automorphism group of the lattice $M$ that leaves the polytope $P$ \eqref{PDef2Eq} invariant.
It is necessarily a finite group%
. In fact, $\Aut P$ is isomorphic to the quotient of the normalizer $N((\CC^*)^n)$ of the
complex torus $(\CC^*)^n$ in $\Aut X$ by $(\CC^*)^n$, so that 
$N((\CC^*)^n)$  consists
of finitely many components each isomorphic to a complex torus \cite[Proposition 3.1]{BS99}.
We note in passing that if $X$ is smooth and Fano, then $\Aut X$ is the product
$$
\Aut X=\Aut P \Aut_0X,
$$
and refer to Proposition \ref{AutClx} and Lemma \ref{PermuteClx} for a proof.

Similarly, we denote by
\begin{equation}\lb{AutDeltaEq}
    \Aut \Delta\subseteq \mathrm{GL}(N)\cong \mathrm{GL}(n,\ZZ)
\end{equation}
the subgroup of the automorphism group of the lattice $N$ that preserves the fan $\Delta$. It is a finite group, since every element in it gives a permutation of $\Delta_1$ \eqref{Delta1Eq} and this embeds $\Aut\Delta$ into the permutation group of $\Delta_1$. 

\begin{remark}
    It is worth noting that the $\Aut P$ we define above may be different from the subgroup of $\mathrm{GL}(M\otimes \mathbb{R})$ that preserves the polytope $P$ when $P$ is not a lattice polytope. And so does $\Aut \Delta$. 

    For example, if $P$ is the convex hull of $\{(0,1),(0,-1),(1/2,0),(-1/2,0)\}$, then linear transformation $$\begin{pmatrix}
        0 & 2 \\
        1/2 & 0
    \end{pmatrix}$$
    preserves $P$, but it is not in $\mathrm{GL}(M)$.
\end{remark}

When $X$ is Fano, the dual map from $\mathrm{GL}(N)$ to $\mathrm{GL}(M)$ induces an isomorphism from $\Aut\Delta$ to $\Aut P$:

\begin{lemma}\label{DualIso}
    If $X$ is Fano, then $\Aut\Delta\cong\Aut P$.
\end{lemma}

\begin{proof}
    For an $\mathcal{P}\in \Aut\Delta$, the dual of it is defined by
    \[\mathcal{P}^*\in \mathrm{GL}(M):\lan m,\mathcal{P}n\ran=\lan\mathcal{P}^*m,n\ran\, \forall m\in M, n\in N \]
    
    For a $\mathcal{P}\in\Aut\Delta$ and any $m\in P\cap M_\RR, \sigma_i\in\Delta_1$, we have
    \[\lan \mathcal{P}^*m,v_i\ran=\lan m,Pv_i\ran\geq -1 \]
    since $Pv_i$ is a primitive vector in $\Delta_1$. This implies that $\mathcal{P}^*m\in P$, which shows $\mathcal{P}^*$ preserves the polytope $P$ and so $\mathcal{P}^*\in\Aut P$.

    Conversely, if $\mathcal{P}^*\in\Aut P$, we want to show that $\mathcal{P}v_i\in\Delta_1$. Because $v_i\in\Delta_1$ and $X$ is Fano, there is a codimension-$1$ facet $F'\sse P$ that satisfies $\lan m,v_i\ran=-1$ for all $m\in F'\cap M_\RR$. Then we have
    \[\lan m,\mathcal{P}v_i\ran=\lan \mathcal{P}^*m,v_i\ran=-1,\, \forall m\in (\mathcal{P}^*)^{-1}F' \]
    And $(\mathcal{P}^*)^{-1}F'$ is another codimension-$1$ facet of $P$ since $\mathcal{P}^*$ preserves the polytope $P$. Thus, $\mathcal{P}v_i$ is another primitive vector in $\Delta_1$, which shows $\mathcal{P}\in \Aut\Delta$. \end{proof}

   There is a special subgroup of $\Aut P$, which will play an important role later. We define:
    \begin{definition}\label{Aut0P}
         $\Aut_0P$ is the subgroup
        $$\{\mathcal{P}\in\Aut P: -F\cap \mathcal{P}F\cap R(P)\neq\emptyset,\, \forall F\hbox{ a codimension-$1$ facet of $P$ satisfying $F\neq\mathcal{P}F$} \}.$$
    \end{definition}

    We will see that (Proposition \ref{WeylGroup}) this is exactly the subgroup of $\Aut P$ that induces the identity map on the divisor class group of $X$.

For every $k\in\NN$, there is an $\Aut X$ action on $H^0(X,-kK_X)$. To get an induced linear action on $M_\QQ$
we must restrict to the normalizer $N((\CC^*)^n)$ of the complex torus $(\CC^*)^n$
in $\Aut X$. The representation of $(\CC^*)^n$ on $H^0(X,-kK_X)$
splits into 1-dimensional spaces, whose generators are called the monomial
basis. There is a one-to-one correspondence between the monomial basis of 
$H^0(X,-kK_X)$ and points in $kP\cap M$, and the  
quotient $N((\CC^*)^n)/(\CC^*)^n$ is a linear group,
that can be identified with $\Aut P$
\eqref{AutPEq}. Since $P$ is defined as the convex hull
of vertices in $M$ it follows that $\Aut P$ is finite. 
Alternatively, this can be seen by observing that 
$N_\RR$ is canonically isomorphic to the quotient of  $(\CC^*)^n$ 
by its
maximal compact subgroup $(S^1)^n$ \cite[p. 229]{BS99}
and the induced action on $M_\RR$ is then defined
by transposing via the pairing.

\subsection{Reductive Lie groups and maximal tori}

We will use some basic concepts and facts about Lie group in this survey. We list them in this subsection without proof. Readers who are interested in the details can refer to standard textbooks on Lie group, such as \cite{Ha15} and \cite{HN12}.


 \begin{definition}
 \label{reductivegp}

\noindent$\bullet\;$ A complex Lie group is a \emph{linear algebraic group} if it can be embedded into $\mathrm{GL}(k, \mbc)$ for some $k$ as a closed subgroup.

\noindent$\bullet\;$ A \emph{faithful representation} of a group $G$ is a representation (that is, a morphism from $G$ to $\mathrm{GL}(V)$ where $V$ is a linear space) with trivial kernel.

\noindent$\bullet\;$
    Let $G$ be a connected linear algebraic group. An element $g\in G$ is called \emph{unipotent} if for a faithful representation $\varphi$, $\varphi(g)-I$ is unipotent (i.e., $(\varphi(g)-I)^k=0$ for some $k>0$). A subgroup of $G$ is called \emph{unipotent} if every element of it is unipotent.

\noindent$\bullet\;$    
    The \textit{unipotent radical} of a connected linear algebraic group $G$ is the largest connected unipotent normal subgroup of $G$. $G$ is called \textit{reductive} if its unipotent radical is trivial.
 \end{definition}

\begin{definition}\label{Weylgp}
    \noindent$\bullet\;$ For a complex Lie group $G$, a \textit{torus} in $G$ is a Lie subgroup isomorphic to $(\mbc^*)^k$. All the tori in $G$ are partially ordered by the inclusion relation. A maximal element under this partial order is called a \textit{maximal torus} of $G$.

    \noindent$\bullet\;$ Let $T$ be a torus in a (connected) reductive complex Lie group $G$, $N(T)$ be its normalizer, then we call
    $$
    W(T):=N(T)/T
    $$
    the \textit{Weyl group} of $G$ with respective to the torus $T$.
\end{definition}

\begin{example}\label{GLWeyl}
    As a simple example, if $G=\mathrm{GL}(n,\mbc)$, we can choose a maximal torus $T$ formed by all diagonal matrices. With respect to this torus, the Weyl group $W$ is the permutation group $S_n$ of the $n$ coordinates. Hence, $W$ can be identified with the subgroup of $G$ consisting of all permutation matrices naturally.
\end{example}

\begin{theorem}[{\cite[Theorem 11.9]{Ha15}}]\label{Torustheo}
    Let $G$ be a linear algebraic group, then all maximal tori in $G$ are conjugate to each other.
    
\end{theorem}

\begin{remark}
    In \cite{Ha15}, the above theorem is proved for compact (real) Lie group. However, combining the fact that any reductive complex Lie group can be obtained from a compact real Lie group by "complexification" (it is well known but nontrivial, see \cite[Theorem 15.3.11]{HN12} and \cite[Theorem 3]{Na61}) and the unipotent radical intersects with any maximal tori trivially, we can translate the results there to the above theorem.
\end{remark}




\section{Symmetry notions for toric Fanos}
\label{Symmetry notions for toric Fanos}

In this section we discuss the classical notions of symmetry of toric Fano manifolds introduced in Definition \ref{SymmDef}. Standard references are \cite[Section 1]{BS99}, \cite[Section 2]{Mab87}, \cite[p.581]{Dem70}, \cite[Section 1]{NP11}.


A few words about how these notions are referred to in the literature.
Batyrev--Selivanova
simply call their notion ``symmetric" \cite[p. 226]{BS99}; Nill
calls centrally lattice  symmetric ``semi-simple" \cite[Section 3]{Nil06}. Mabuchi refers to $R(P)$ as the roots of $M$ with respect to the fan $\Delta$ \cite[Section 2]{Mab87}.

Let us illustrate these notions with some examples. 
Example \ref{BSsymmnotsymmExam} shows that these three notions are not all identical.
However, Examples \ref{BSsymmnotsymmExam}--\ref{PtwoallsymmExam} leave open the possibility that Batyrev--Selivanova
symmetry is the same as central lattice  symmetry.

\bexam
\lb{BSsymmnotsymmExam}
For $\PP^2$, $\Aut P=S_3$ so $P$ is Batyrev--Selivanova
symmetric (see Figure \ref{Fig for R(P)}). It is also centrally lattice  symmetric since
$$R(P)=\{(\pm1,0),(0,\pm1),(\mp1,\pm1)\}=-R(P).$$
But $P$ is not centrally symmetric, i.e., $P\not=-P$. 
\eexam

\bexam
\lb{PtwoallsymmExam}
For $\PP^2$ blown-up at a point,
$$R(P)=\{(\mp1,\pm1),(1,0),(0,1)\}\not=-R(P).$$
Also $P\not=-P$. Finally, $P$  is not Batyrev--Selivanova
symmetric as $\Aut P$ fixes all points $\{(n,n)\,:\, n\in\ZZ\}$
(see Figure \ref{Figure of blowup 1,2 points}).
\eexam

In fact, the three symmetry notions of Definition \ref{SymmDef} are all distinct. They are related as follows.

\begin{proposition}
\label{relation of three symm notions}
The relationship between the symmetry notions in Definition \ref{SymmDef} is:\hfill\break

\noindent $P$ is centrally symmetric     $\Rightarrow$
$P$ is Batyrev--Selivanova symmetric     $\Rightarrow$
$P$ is centrally lattice symmetric.
Moreover, none of the implications can be reversed.
\end{proposition}

\bremark
For the blow-up of $\PP^2$ at three non-colinear points, 
$$P=\co\{(\pm1,0),(0,\pm1),(\mp1,\pm1)\}=-P,$$ 
while $R(P)=\emptyset$ (see Figure \ref{Fig for R(P)}). This is fine, since we 
adopt the convention that the empty set satisfies all the symmetry notions
of Definition \ref{SymmDef}. 
\eremark

The final part of
Proposition \ref{relation of three symm notions}, specifically the non-equivalence
of 
Batyrev--Selivanova
symmetry and central lattice symmetry
turns out to be the most subtle part of the theorem,
and was an open problem posed by Batyrev--Selivanova \cite[p. 227]{BS99}
and solved by \cite[Proposition 1.4]{NP11}.

Our point of view is that it is more natural to prove 
Proposition \ref{relation of three symm notions}
by proving a more general theorem (Theorem \ref{relation of symm notions}) that includes several more equivalent notions of symmetry. 
Before doing that though, let us, as a warm-up, prove the easiest implication of Proposition \ref{relation of three symm notions}:

\begin{lemma}
\label{(i)=>(ii)}
    If $P$ is centrally symmetric, then it is Batyrev--Selivanova symmetric.
\end{lemma}
\begin{proof}
Central symmetry of $P$ implies that $-\id\in\Aut P$. 
 So if a point $u$ is fixed by $\Aut P$, in particular it is fixed by the central reflection, i.e., $u=-u$. Thus $u=0$. So $P$ is Batyrev--Selivanova symmetric.
\end{proof}





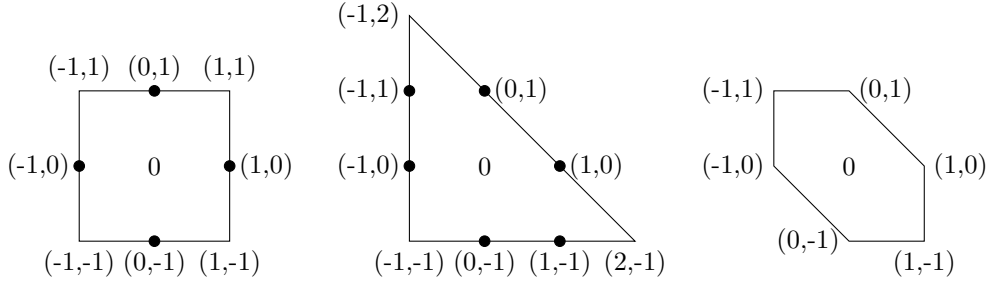
\begin{figure}
    \centering
    \begin{subfigure}{.3\textwidth}
        \centering
        \begin{tikzpicture}
            \draw(1,-1)node[below]{(1,-1)}--(1,1)node[above]{(1,1)}--(-1,1)node[above]{(-1,1)}--(-1,-1)node[below]{(-1,-1)}--cycle;
            \filldraw(1,0)circle(2pt)node[right]{(1,0)};
            \filldraw(0,1)circle(2pt)node[above]{(0,1)};
            \filldraw(-1,0)circle(2pt)node[left]{(-1,0)};
            \filldraw(0,-1)circle(2pt)node[below]{(0,-1)};
            \draw(0,0)node{$0$};
        \end{tikzpicture}
    \end{subfigure}
    \begin{subfigure}{.3\textwidth}
        \centering
        \begin{tikzpicture}
            \draw(2,-1)node[below]{(2,-1)}--(-1,2)node[left]{(-1,2)}--(-1,-1)node[below]{(-1,-1)}--cycle;
            \filldraw(1,0)circle(2pt)node[right]{(1,0)};
            \filldraw(0,1)circle(2pt)node[right]{(0,1)};
            \filldraw(-1,1)circle(2pt)node[left]{(-1,1)};
            \filldraw(-1,0)circle(2pt)node[left]{(-1,0)};
            \filldraw(0,-1)circle(2pt)node[below]{(0,-1)};
            \filldraw(1,-1)circle(2pt)node[below]{(1,-1)};
            \draw(0,0)node{$0$};
        \end{tikzpicture}
    \end{subfigure}
    \begin{subfigure}{.3\textwidth}
        \centering
        \begin{tikzpicture}
            \draw(1,-1)node[below]{(1,-1)}--(1,0)node[right]{(1,0)}--(0,1)node[right]{(0,1)}--(-1,1)node[left]{(-1,1)}--(-1,0)node[left]{(-1,0)}--(0,-1)node[left]{(0,-1)}--cycle;
            \draw(0,0)node{$0$};
        \end{tikzpicture}
    \end{subfigure}
    \caption{\small The polytope $P$ corresponding to the K-polystable toric del Pezzo surfaces, i.e., $\PP^1\times\PP^1$, $\PP^2$ and $\PP^2$ blown-up at three non-linear points. The points of $R(P)$ are marked along the boundary of $P$. Notice that the automorphism group of $P$ is isomorphic to the automorphism group of the polygon with $\ell$ vertices, i.e., $D_{2\ell}\subset S_\ell$, the dihedral group of order $2\ell$; specifically, $\Aut P\cong D_8,D_6=S_3,D_{12}$, respectively, generated by cyclic permutations and 
    reflections. In particular, they are Batyrev--Selivanova symmetric.}
    \label{Fig for R(P)}
\end{figure}
\begin{figure}
    \centering
    \begin{subfigure}{.3\textwidth}
        \centering
        \begin{tikzpicture}
            \draw(2,-1)node[below]{(2,-1)}--(-1,2)node[left]{(-1,2)}--(-1,0)node[left]{(-1,0)}--(0,-1)node[below]{(0,-1)}--cycle;
            \filldraw[red](1,0)circle(2pt)node[right]{(1,0)};
            \filldraw[red](0,1)circle(2pt)node[right]{(0,1)};
            \filldraw(-1,1)circle(2pt)node[left]{(-1,1)};
            \filldraw(1,-1)circle(2pt)node[below]{(1,-1)};
            \draw(0,0)node{$0$};
        \end{tikzpicture}
    \end{subfigure}
    \begin{subfigure}{.3\textwidth}
        \centering
        \begin{tikzpicture}
            \draw(1,-1)node[below]{(1,-1)}--(1,0)node[right]{(1,0)}--(0,1)node[right]{(0,1)}--(-1,1)node[left]{(-1,1)}--(-1,-1)node[left]{(-1,-1)}--cycle;
            \filldraw[red](-1,0)circle(2pt)node[left]{(-1,0)};
            \filldraw[red](0,-1)circle(2pt)node[below]{(0,-1)};
            \draw(0,0)node{$0$};
        \end{tikzpicture}
    \end{subfigure}
    \caption{\small The polytope $P$ corresponding to the K-unstable toric del Pezzo surfaces, i.e., $\PP^2$ blown-up at one or two points. The points of $R(P)$ are marked along the boundary of $P$, and the points in $R_u(P)$ are marked in red. Since automorphisms of the lattice preserve the normalized volumes of facets, the only non-trivial automorphism of $P$ is the reflection over the line $y=x$. Therefore, $\Aut P\cong\ZZ/2\ZZ$, generated by $\begin{pmatrix}0&1\cr1&0\cr\end{pmatrix}$. In particular, they are not Batyrev--Selivanova symmetric since $\Aut P$ fixes the set
    $\{(x,x)\,:\, x\in\ZZ\}\not=\{(0,0)\}$ in $M$. Since $\Bc_k(P)$ and $\Bc(P)$ are fixed by $\Aut P$, they must lie on the line $y=x$.}
    \label{Figure of blowup 1,2 points}
\end{figure}
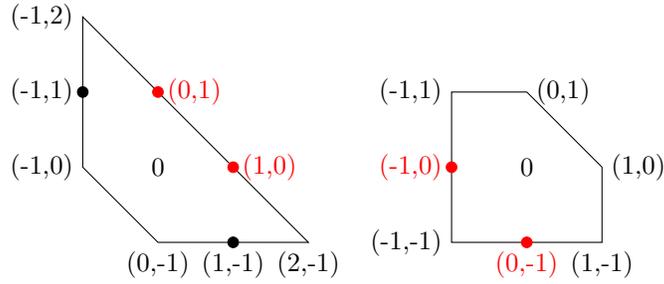

\section{Notions of barycenters}\label{Notions of barycenters}

In this section we discuss two different notions of barycenters for lattice polytopes. Let 
$$M\cong\ZZ^n$$ 
denote a lattice and $P\subset M\otimes_\ZZ\RR\cong\RR^n$
denote a convex lattice polytope, i.e., the vertex set
satisfies $\Ver P\subset M$. 
The {\it $k$-th quantized barycenter} (or $k$-th discrete barycenter) of $P$
is defined as the average of the lattice points in $kP$ divided by $k$:
$$
\Bc_k(P):=\frac1{k|kP\cap M|}\sum_{u\in kP\cap M}u, \q k\in\NN.
$$
The barycenter of $P$ 
is the first moment of the uniform probability measure on $P$,
considered as a convex body in 
$M_\RR:=M\otimes_\ZZ\RR\cong\RR^n$, or
$$
\Bc(P):=\frac1{|P|}\int_Pxdx\in M_\RR.
$$
The empirical measures $\frac1{|kP\cap M|}\sum_{u\in P\cap k^{-1}M}\delta_u$
converge weakly to $1_Pdx/|P|$ so
$$
\Bc(P)=
\lim_{k\ra\infty}\Bc_k(P)\in M_\RR.
$$

We can count the lattice points in $kP$ thanks to Ehrhart theory.
\begin{theorem}{\rm \cite{Ehr67a,Ehr67b}, \cite[Theorem 19.1]{Gru07},\cite[Theorem 12.2, Exercise 12.5]{MS05},\cite[Corollary 5.5]{BR07}}
\label{Ehrhart polynomial}
    Let $M$ be a lattice of dimension $n$. For any lattice polytope $P$, there is a polynomial
    $$
        E_P(k)=\sum_{i=0}^na_ik^i
    $$
    such that
    $$
        E_P(k)=\#(kP\cap M)
    $$
    for any $k\in\NN$. In particular,
    $$\begin{aligned}
        a_n&=\Vol(P),&a_0&=1.
    \end{aligned}$$
\end{theorem}

\begin{theorem}{\rm(\cite[Theorem 2.3]{JR2})}\label{JR2thm}
    If $\Bc_k(P)=0$ for $n+1$ values of $k$, then $\Bc_k(P)=0$ for all $k$ and $\Bc(P)=0$.
\end{theorem}
\begin{proof}
    \begin{figure}
        \centering
        \begin{tikzpicture}
            \filldraw[fill=gray!50](0,0)--(4,0)--(5,{sqrt(3)})--(5,{4+sqrt(3)})--(4,4)--cycle;
            \filldraw[fill=gray!25](0,0)--(4,4)--(5,{4+sqrt(3)})--cycle;
            \draw(4,0)--(4,4);
            \draw(4.75,{3*sqrt(3)/4}); 
            \draw(5,{2/3+sqrt(3)});
            \filldraw(0,0)circle(2pt)node[left]{$u_i=-C_i$};
            \filldraw({4/3},0)circle(2pt);
            \filldraw({8/3},0)circle(2pt);
            \filldraw(4,0)circle(2pt)node[right]{$u_i=3-C_i$};
            \filldraw(4.5,{sqrt(3)/2})circle(2pt);
            \filldraw(5,{sqrt(3)})circle(2pt);
            \filldraw({4/3},{4/3})circle(2pt);
            \filldraw({8/3},{4/3})circle(2pt);
            \filldraw(4,{4/3})circle(2pt);
            \filldraw(4.5,{4/3+sqrt(3)/2})circle(2pt);
            \filldraw(5,{4/3+sqrt(3)})circle(2pt);
            \filldraw({8/3},{8/3})circle(2pt);
            \filldraw(4,{8/3})circle(2pt);
            \filldraw(4.5,{8/3+sqrt(3)/2})circle(2pt);
            \filldraw(5,{8/3+sqrt(3)})circle(2pt)node[right]{$\;\;\;\;P_i$};
            \filldraw(4,4)circle(2pt);
            \filldraw(4.5,{4+sqrt(3)/2})circle(2pt);
            \filldraw(5,{4+sqrt(3)})circle(2pt)node[right]{$h=3$};
            \filldraw({8/3+.5},{8/3+sqrt(3)/2})circle(2pt);
            \draw[->](3,-.5)--(3,-1.5)node[midway,right]{$\pi$};
            \filldraw[fill=gray!25](0,-3)--(4,-3)--(5,{sqrt(3)-3})--cycle;
            \filldraw(0,-3)circle(2pt)node[left]{$u_i=-C_i$};
            \filldraw({4/3},-3)circle(2pt);
            \filldraw({8/3},-3)circle(2pt);
            \filldraw(4,-3)circle(2pt)node[right]{$u_i=3-C_i$};
            \filldraw(4.5,{sqrt(3)/2-3})circle(2pt);
            \filldraw(5,{sqrt(3)-3})circle(2pt);
            \filldraw({8/3+.5},{-3+sqrt(3)/2})circle(2pt);
            \draw[dotted]({4/3},0)--({4/3},{4/3});
            \draw[dotted]({8/3},0)--({8/3},{8/3});
            \draw[dotted](4.5,{sqrt(3)/2})--(4.5,{sqrt(3)/2+4});
        \end{tikzpicture}
        \caption{The lattice polytope $P_i$ has $u_i+C_i+1$ lattice points above each lattice
        point in $P$ with $i$-th coordinate equal to $u_i$. The dotted lines are fibers of $\pi:M\times\RR\ni (u,h)\mapsto u\in M$.}\label{lattice polytope}
    \end{figure}
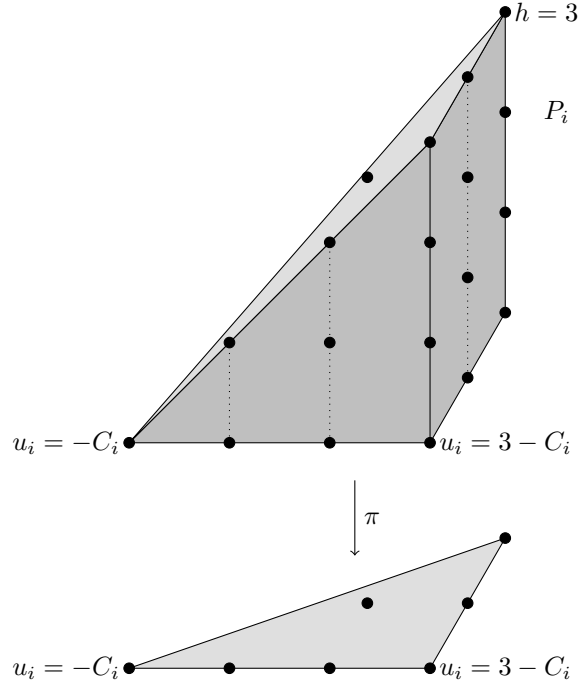
    Fix $1\leq i\leq n$. Pick an integer $C_i\geq-\min\{u_i:(u_1,\ldots,u_n)\in P\}$. Then for any $u\in P$, the function $h_i(u):=u_i+C_i$ is non-negative. Consider the lattice polytope
    (see Figure \ref{lattice polytope})
    \begin{equation}\label{no translation}
        P_i:=\{(u,h)\in\RR^n\times\RR~:~u\in P,~0\leq h\leq h_i(u)\}.
    \end{equation}
    Notice that
    \begin{align*}
        E_{P_i}(k)&=\sum_{u\in P\cap\frac{1}{k}M}\#\left\{\left(u,0\right),\left(u,\frac{1}{k}\right),\cdots,\left(u,h_i(u)\right)\right\}\\
        &=\sum_{u\in P\cap\frac{1}{k}M}(k(u_i+C_i)+1)\\
        &=(kC_i+1)E_P(k)+k\sum_{u\in P\cap\frac{1}{k}M}u_i.
    \end{align*}
    Therefore,
    $$
    \Bc_{k,i}\left(P\right)
        =\frac{1}{E_P(k)}\sum_{u\in P\cap\frac{1}{k}M}u_i
        =\frac{E_{P_i}(k)-(C_ik+1)E_P(k)}{kE_P(k)}.\label{Bc_k no translation}
    $$
    Notice that the numerator
    $$
        Q_i(k):=E_{P_i}(k)-(C_ik+1)E_P(k)
    $$
    is a polynomial with respect to $k$ of degree at most $n+1$ and with constant term $0$ by Theorem \ref{Ehrhart polynomial}. Hence $Q_i(k)=k\widetilde{Q}_i(k)$ for some polynomial $\widetilde{Q}_i$ of degree at most $n$.
    Now,
    $$
        \Bc_{k,i}\left(P\right)=\frac{\widetilde{Q}_i\left(k\right)}{E_P\left(k\right)}.
    $$
    Since $\Bc_{k,i}$ vanishes for at least $n+1$ values of $k$, $\widetilde{Q}_i\equiv0$, i.e., $\Bc_{k,i}(P)=0$ for all $k$.
\end{proof}

\section{Notions of canonical metrics}\label{Notions of canonical metrics}

In this section we discuss various notions of canonical metrics on toric manifolds.

\begin{definition}\label{KE def}
    A Fano manifold admits a K\"ahler--Einstein metric if there is a K\"ahler form $\omega$ such that $\Ric\,\omega=\omega$.
\end{definition}

\begin{remark}\label{existence of KE}
    Recall that for toric Fano manifolds, the existence of K\"ahler--Einstein metrics is equivalent to $\Bc(P)=0$ \cite[Corollary 1.2]{WZ04}.
    By Matsushima \cite{Mat57}, the existence of K\"ahler--Einstein metrics implies that $\Aut_0X$ is reductive. 
\end{remark}

\subsection{Balanced metrics}

The following notion of balanced metrics goes back to Zhang \cite{Zh96} and Luo \cite{Luo98}, and was popularized by Donaldson \cite[p.\ 481]{Don01}.
\begin{definition}
    Suppose $L$ is an ample line bundle. Let
    $$
        \cH_k=\left\{\text{positively curved Hermitian metrics on $-kL$}\right\},
    $$
    and
    $$
        \cB_k=\left\{\text{Hermitian inner products on }H^0\left(X,-kL\right)\right\}.
    $$
    Fix
    \begin{equation}\label{reference}
        h\in\cH_1,\qquad\omega:=-\sqrt{-1}\partial\bar\partial\log h.
    \end{equation}
    For $\varphi$ such that $\omega_\varphi:=\omega+i\partial\bar\partial\varphi>0$, let $h_\varphi:=he^{-\varphi}$.
    Then
    $$
        \cH_k=\left\{h_\varphi^k\,:\,\omega_\varphi>0\right\}.
    $$
    Let $\Hilb_k:\cH_k\to\cB_k$,
    $$
        \Hilb_k\left(h_\varphi^k\right)\left(s,t\right):=\frac{d_k}{\int_X\omega_\varphi^n}\int_Xh_\varphi^k\left(s,t\right)\omega_\varphi^n,
    $$
    where $d_k=h^0(X,-kL)$. Suppose $kL$ is very ample. For $H\in\cB_k$, $\FS_k(H)\in\cH_k$ is given by
    \begin{equation}\label{FSk}
        \FS_k\left(H\right):=\left(\sum_i\left|s_i\right|^2\right)^{-1},
    \end{equation}
    i.e., for $s,t\in H^0(X,-kL)$,
    $$
        \left\langle s,t\right\rangle_{\FS_k\left(H\right)}=\frac{s\bar{t}}{\sum_i\left|s_i\right|^2}.
    $$
    where $\{s_i\}$ is an(y) $H$-orthonormal basis.
    
    A Hermitian metric $h_\varphi^k\in\cH_k$ is \textit{balanced} if
    $$
        h_\varphi^k=\FS_k\circ\Hilb_k\left(h_\varphi^k\right).
    $$
\end{definition}

According to a landmark result of Donaldson, balanced metrics approximate constant scalar curvature metrics \cite[Theorems 2 and 3]{Don01}.


Balanced metrics are related to asymptotics of Bergman kernel and the Mabuchi functional.

\subsection{{\RTZ} metrics}

Another notion of balanced metics related to the Ding functional was introduced by Rubinstein--Tian--Zhang \cite[Lemma 4.2]{RTZ21}.

\begin{definition}
    Recall \eqref{reference}. For any $f\in C^\infty(X)$, the map $\Hilb_k^{f,h}:\cH_k\to\cB_k$ is given by
    $$
        \Hilb_k^{f,h}\left(h_\varphi^k\right)\left(s,t\right):=\frac{d_k}{\int_Xe^{f-\varphi}\omega^n}\int_Xh_\varphi^k\left(s,t\right)e^{f-\varphi}\omega^n.
    $$
    
    A Hermitian metric $h_\varphi^k\in\cH_k$ is \textit{\RTZ} if
    $$
        h_\varphi^k=\FS_k\circ\Hilb_k^{f,h}\left(h_\varphi^k\right).
    $$
\end{definition}

The key observation of Rubinstein--Tian--Zhang was that balanced metrics are related to the Ding functional.
\begin{definition}
    Fix $K\in\cB_k$. For $h_\varphi^k=\FS_k(H)$, the quantized Ding functional is given by
    $$
        F_k^{f,h,K}\left(\varphi\right):=-\log\frac{1}{\int_X\omega^n}\int_Xe^{f-\varphi}\omega^n+\frac{1}{kd_k}\log\det K^{-1}H.
    $$
\end{definition}
\begin{lemma}{\rm\cite[Lemma 4.2]{RTZ21}}
    The metric $h_\varphi^k=\FS_k(H)$ is {\RTZ} if and only if $\varphi$ is a critical point of $F_k^{f,h,K}$.
\end{lemma}
\begin{proof}
    For $h_\varphi^k=\FS_k(H)$, consider a curve $t\mapsto H_t$ with $H_0=H$. For a basis $\{s_i\}$ of $H^0(X,kL)$, let $M=(M_{ij})$ be given by
    $$
        M_{ij}=\left.\frac{d}{dt}\right|_{t=0}H_t\left(s_i,s_j\right).
    $$
    By simultaneous diagonalization, we may choose an $H$-orthonormal basis $\{s_i\}$ such that $M$ is diagonal, say $M=\diag(\lambda_1,\ldots,\lambda_{d_k})$. Then
    $$
        \left.\frac{d}{dt}\right|_{t=0}\log\det K^{-1}H_t=\sum_i\lambda_i=\tr M.
    $$
    Let $h_{\varphi_t}^k=\FS_k(H_t)$. By \eqref{FSk},
    \begin{align*}
        \left.\frac{d}{dt}\right|_{t=0}\varphi_t&=\left.\frac{d}{dt}\right|_{t=0}\left(-\frac{1}{k}\log\FS_k\left(H_t\right)\right)\\
        &=-\frac{\sum_i\lambda_i\left|s_i\right|^2}{k\sum_i\left|s_i\right|^2}\\
        &=-\frac{1}{k}\sum_i\lambda_ih_\varphi^k\left(s_i,s_i\right)\\
        &=-\frac{1}{k}\sum_{i,j}M_{ij}h_\varphi^k\left(s_i,s_j\right).
    \end{align*}
    Therefore,
    \begin{align*}
        \left.\frac{d}{dt}\right|_{t=0}F_k^{f,h,K}\left(\varphi_t\right)&=-\frac{1}{\int_Xe^{f-\varphi}\omega^n}\int_X\frac{1}{k}\sum_{i,j}M_{ij}h_\varphi^k\left(s_i,s_j\right)e^{f-\varphi}\omega^n+\frac{1}{kd_k}\tr M\\
        &=\sum_i\frac{M_{ij}}{kd_k}\left(\delta_{ij}-\frac{d_k}{\int_Xe^{f-\varphi}\omega^n}\int_Xh_\varphi^k\left(s_i,s_j\right)e^{f-\varphi}\omega^n\right)\\
        &=\sum_i\frac{M_{ij}}{kd_k}\left(\delta_{ij}-\Hilb_k^{f,h}\left(h_\varphi^k\right)\left(s_i,s_j\right)\right).
    \end{align*}
    Now, $\varphi$ is a critical point of $F_k^{f,h,K}$ if and only if for any $M$, $\frac{d}{dt}|_{t=0}F_k^{f,h,K}(\varphi_t)=0$, i.e.,
    $$
        \Hilb_k^{f,h}\left(h_\varphi^k\right)\left(s_i,s_j\right)=\delta_{ij}.
    $$
    In other words, $\{s_i\}$ is a $\Hilb_k^{f,h}(h_\varphi^k)$-orthonormal basis. Equivalently, since $\{s_i\}$ is also $H$-orthonormal,
    $$
        H=\Hilb_k^{f,h}\left(h_\varphi^k\right),
    $$
    and
    \[
        h_\varphi^k=\FS_k\left(H\right)=\FS_k\left(\Hilb_k^{f,h}\left(h_\varphi^k\right)\right).\qedhere
    \]
\end{proof}

\subsection{Anticanonically balanced metrics}

The notion of {\RTZ} metrics depends on the choice of $(f,h)$
and is defined for any polarization $L$. In the special case when $L=-K_X>0$, i.e., $X$ is anticanonically polarized Fano,  there is a canonical choice.
The next definition is due to 
Donaldson \cite[p.\ 588]{Don09}.
\begin{definition}\label{abm def}
    Suppose $X$ is Fano and let $L=-K_X$. Notice that for any $k$ and $h^k\in\cH_k$, $\mu_h:=h$ is a volume form on $X$. The map $\HilbA:\cH_k\to\cB_k$ is given by
    $$
        \HilbA\left(h_\varphi^k\right)\left(s,t\right):=\frac{d_k}{\int_X\mu_{h_\varphi}}\int_Xh_\varphi^k\left(s,t\right)\mu_{h_\varphi},
    $$
    
    A Hermitian metric $h_\varphi^k\in\cH_k$ is \textit{anticanonically balanced} if
    $$
        h_\varphi^k=\FS_k\circ\HilbA\left(h_\varphi^k\right).
    $$
\end{definition}

\begin{lemma}
    Recall \eqref{reference}. Let $f_h$ be such that
    \begin{equation}\label{Ricci}
        e^{f_h}\omega^n=\mu_h.
    \end{equation}
    In other words, $f_h$ is the Ricci potential of $\omega$, defined (up to a constant) by
    \begin{equation}\label{Ricci potential}
        \sqrt{-1}\partial\bar\partial f_h=\Ric\omega-\omega.
    \end{equation}
    Then
    $$
        \HilbA=\Hilb_k^{f_h,h}.
    $$
    In particular, anticanonically balanced metrics are {\RTZ} metrics.
\end{lemma}
\begin{proof}
    First, note that applying $\sqrt{-1}\partial\bar\partial\log$ to \eqref{Ricci} gives
    $$
        \sqrt{-1}\partial\bar\partial f_h-\Ric\omega=\sqrt{-1}\partial\bar\partial\log h=-\omega.
    $$
    So \eqref{Ricci potential} holds.
    
    Next, by \eqref{Ricci},
    \begin{align*}
        \Hilb_k^{f_h,h}\left(h_\varphi\right)\left(s,t\right)&=\frac{1}{\int_Xe^{f_h-\varphi}\omega^n}\int_Xh_\varphi^k\left(s,t\right)e^{f_h-\varphi}\omega^n\\
        &=\frac{1}{\int_Xe^{-\varphi}\mu_h}\int_Xh_\varphi^k\left(s,t\right)e^{-\varphi}\mu_h\\
        &=\frac{1}{\int_X\mu_{h_\varphi}}\int_Xh_\varphi^k\left(s,t\right)\mu_{h_\varphi}\\
        &=\HilbA\left(h_\varphi\right)\left(s,t\right).\qedhere
    \end{align*}
\end{proof}

\section{Fundamental theorem on symmetry of toric Fanos}\label{Section Fundamental}

\subsection{Relating symmetry to barycenters}

The notions introduced in \S\ref{Symmetry notions for toric Fanos} and \S\ref{Notions of barycenters} are related as follows.
Some of the implications are well-known (but scattered across the literature).

\begin{theorem}\label{relation of symm notions}
The relationship of the following notions are that 
\textup{(\ref{item central sym})}
$\Rightarrow$
\textup{(\ref{item BS sym})}
$\Rightarrow$
\textup{(\ref{item Bc_k=0 for all k})}
$\Leftrightarrow$
\textup{(\ref{item Bc_k=0 for large k})}
$\Leftrightarrow$
\textup{(\ref{item Bc_k=0 for n+2 values of k})}
$\Rightarrow$
\textup{(\ref{item Bc=0})}
$\Rightarrow$
\textup{(\ref{item R(P)=-R(P)})}.
\begin{enumerate}[$(\rm i)$]
\item \label{item central sym}
$P$ is centrally symmetric; 
\item \label{item BS sym}
$P$ is Batyrev--Selivanova symmetric; 
\item \label{item Bc_k=0 for all k}
$\Bc_k(P)=0$ for all $k\in\NN$;
\item  \label{item Bc_k=0 for large k}
$\Bc_k(P)=0$ for all sufficiently large $k\in\NN$;
\item  \label{item Bc_k=0 for n+2 values of k}
$\Bc_k(P)=0$ for at least $n+1$ values of $k\in\NN$, 
\item  \label{item Bc=0}
$\Bc(P)=0$;
\item \label{item R(P)=-R(P)}
$P$ is centrally lattice symmetric.
\end{enumerate}
\end{theorem}

\begin{proof}

 
(\ref{item central sym})$\Rightarrow$(\ref{item BS sym}): This is Lemma \ref{(i)=>(ii)}.

\smallskip\noindent(\ref{item BS sym})$\Rightarrow$(\ref{item Bc_k=0 for all k}):
Since $\Aut P$ is a subgroup of the {\it linear} group $\Aut M\cong\mathrm{GL}(n, \ZZ)$, it follows from Definition \ref{SymmDef} that not only is $P\cap M$ invariant
under $\Aut P$, but so is $kP\cap M$ for each $k\in\NN$. 
That is, for each $g\in\Aut P$, 
one has
$g.(kP\cap M)=kP\cap M$.
In particular, $\sum_{u\in kP\cap M}u\in M$ is fixed by $\Aut P$. Since $P$ is Batyrev--Selivanova symmetric, such fixed point has to be $0\in M$, i.e., $\sum_{u\in kP\cap M}u=0$. Hence,
$$\Bc_k(P)=\frac{1}{\#(kP\cap M)}\sum_{u\in kP\cap M}u=0.$$

\smallskip\noindent(\ref{item Bc_k=0 for all k})$\Leftrightarrow$(\ref{item Bc_k=0 for large k})$\Leftrightarrow$(\ref{item Bc_k=0 for n+2 values of k})$\Rightarrow$(\ref{item Bc=0}): This follows from Theorem \ref{JR2thm}.

\smallskip\noindent(\ref{item Bc=0})$\Rightarrow$(\ref{item R(P)=-R(P)}):
A purely coombinatoric proof is due to Nill \cite[Theorem 1.5]{Nil06}. We give here an alternative proof using toric geometry. By Remark \ref{existence of KE}, $\Aut_0X$ is reductive. The proof is done once we establish the following result:

\begin{lemma}\label{lattice-centrally-symm-reductive}
    If X is complete and simplicial (equivalently, a compact orbifold), then $P$ is centrally lattice symmetric if and only if $\Aut_0X$ is reductive.
\end{lemma}

Lemma \ref{lattice-centrally-symm-reductive} is a corollary of Demazure's Structure Theorem \ref{Demazure's structure theorem}, 
proved in section \ref{Proof of Demazure's Theorem}.
Thus, we postpone the proof of the lemma to \S\ref{SymEqualRed}.
\end{proof}



\subsection{Relation to stability thresholds}

The equivalence \eqref{item BS sym-intro}$\Leftrightarrow$\eqref{item alphaGone-intro} in Theorem \ref{relation of symm notions on metric level} is due to Song \cite[Corollary 1.1]{Son05}. Note that this is the special case of \cite[Theorem 1.4]{JR1} and Demailly's theorem \cite[(A.1)]{CS08}. Cheltsov--Shramov claimed a more general formula without the real torus symmetry. However (as kindly pointed out to us by I. Cheltsov) there is an error in the proof of \cite[Lemma 5.1]{CS08} as the toric degeneration used there need not respect the automorphism. Below we give a proof following \cite{JR1}.

\begin{theorem}{\rm\cite[Theorem 1.4]{JR1}}\label{Song}
    For any subgroup $H$ of $\Aut P$, let $G(H)=H\ltimes(S^1)^n$. Then
    $$
        \alpha_{G\left(H\right)}=\sup\left\{c\in(0,1)\,:\,-\frac{c}{1-c}P^H\subset P\right\},
    $$
    where $P^H$ is the space of fixed points of $P$ under $H$.
\end{theorem}
\begin{proof}[Sketch of proof of Theorem \ref{Song}]
    By Demailly's theorem \cite[(A.1)]{CS08},
    \begin{equation}\label{Demailly}
        \alpha_{G(H)}=\inf_{k\in\NN}\alpha_{k,G(H)},
    \end{equation}
    where
    \begin{equation}\label{glct}
        \alpha_{k,G(H)}=k\inf_{\substack{|V|\subset|-kK_X|\\V^{G(H)}=V}}\lct|V|,
    \end{equation}
    and
    $$
        \lct|V|
        :=\sup\left\{c>0\,:\,~ \left(\sum_j|\nu_j(z)|^2\right)^{-{c}} \text{ is locally integrable on }X\text{ for a(ny) basis }\{\nu_j\}\text{ of }V\right\}.
    $$
    Now, since any $(S^1)^n$-invariant linear system $|V|$ is spanned by monomials (see \cite[Lemma 4.8]{JR1}), \eqref{glct} reduces to
    \begin{align*}
        \alpha_{k,G(H)}&=\sup\left\{c>0\,:\,\int_X\left(\sum_{u\in O^{(k)}_i}\left|s_{k,u}\right|_{h_k}^2\right)^{-\frac{c}{k}}d\mu<\infty,~\forall~1\leq i\leq N\right\}\\
        &=\sup\left\{c>0\,:\,\int_X\Big(\sum_{\sigma\in H}\left|s_{k,\sigma u}\right|_{h_k}^2\Big)^{-\frac{c}{k}}d\mu<\infty,~\forall~u\in P\cap M/k\right\},
    \end{align*}
    where $O^{(k)}_1,\ldots,O^{(k)}_N$ are the orbits of the action of $H$ on $k^{-1}M\cap P$, $s_{k,u}$ is the monomial section in $H^0(X,-kK_X)$ corresponding to the point $u\in k^{-1}M\cap P$, $h_k=h^{\otimes k}$ for a(ny) fixed metric $h$ on $-K_X$, and $d\mu$ is a(ny) smooth volume form on $X$.

    By \eqref{Demailly},
    \begin{equation}\label{alpha_G}
        \alpha_{G(H)}=\sup\left\{c>0\,:\,\int_X\Big(\sum_{\sigma\in H}\left|s_{k,\sigma u}\right|_{h_k}^2\Big)^{-\frac{c}{k}}d\mu<\infty,~\forall~u\in P\cap M/k,~k\in\NN\right\}.
    \end{equation}

    Now, observe that by the geometric-arithmetic mean inequality,
    \begin{align*}
        \left(\frac{1}{|H|}\sum_{\sigma\in H}\left|s_{k,\sigma u}\right|_{h_k}^2\right)^{-\frac{c}{k}}\leq\prod_{\sigma\in H}\left|s_{k,\sigma u}\right|_{h_k}^{-\frac{2c}{|H|k}}
        &=\left|s_{|H|k,\pi_H(u)
        }\right|_{h_{|H|k}}^{-\frac{2c}{|H|k}}
        \cr
        &=\left(\frac{1}{|H|}\sum_{\sigma\in H}\left|s_{|H|k,\sigma\pi_H(u)
        }\right|_{h_{|H|k}}^2\right)^{-\frac{c}{|H|k}},
    \end{align*}
    where $\pi_H:=\frac{1}{|H|}\sum_{\eta\in H}\eta\in
    \hbox{End}(M_\QQ)$ is the map that takes a point in $M_\RR$ to the average of its $H$-orbit, since $\pi_H(u)$ is fixed by $H$. Therefore the supremum in \eqref{alpha_G} is unchanged if restricted to those $u$ fixed by $H$. Thus
    $$
        \alpha_{G(H)}=\inf\left\{k\cdot\lct\left(s_{k,u}\right)\,:\,~u\in P^H\cap\frac{1}{k}M,~k\in\NN\right\}.
    $$
    Recall \cite[Corollary 7.4]{BJ20},
    $$
        k\cdot\lct\left(s_{k,u}\right)=\frac{1}{1+\left\|u\right\|_{-P}}.
    $$
    We have
    \begin{equation}\label{sketch}
        \alpha_{G(H)}=\inf\left\{\frac{1}{1+\left\|u\right\|_{-P}}\,:\,~u\in P^H\cap M_\QQ\right\}=\frac{1}{1+\sup\limits_{P^H\cap M_\QQ}\left\|\,\cdot\,\right\|_{-P}}=\frac{1}{1+\max\limits_{P^H}\left\|\,\cdot\,\right\|_{-P}}.
    \end{equation}
    Equivalently, since
    $$
        \left\|u\right\|_{-P}=\inf\left\{\lambda>0\,\middle|\,-\frac{u}{\lambda}\in P\right\},
    $$
    we have
    $$
        \frac{1}{1+\left\|u\right\|_{-P}}=\sup\left\{c\in\left(0,1\right)\,\middle|\,-\frac{c}{1-c}u\in P\right\},
    $$
    and we can rewrite \eqref{sketch} as
    $$
        \alpha_{G(H)}=\sup\left\{c\in(0,1)\,:\,-\frac{c}{1-c}P^H\subset P\right\}.
    $$
\end{proof}

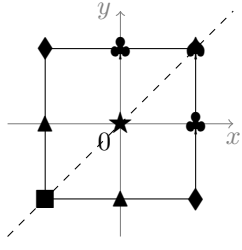
\begin{figure}
    \centering
    \begin{tikzpicture}
        \draw[->,gray](-1.5,0)--(1.5,0)node[below]{$x$};
        \draw[->,gray](0,-1.5)--(0,1.5)node[left]{$y$};
        \draw(0,0)node[below left]{$0$};
        \draw(1,0)node{$\clubsuit$}--(1,1)node{$\spadesuit$}--(0,1)node{$\clubsuit$}--(-1,1)node{$\blacklozenge$}--(-1,0)node{$\blacktriangle$}--(-1,-1)node{$\blacksquare$}--(0,-1)node{$\blacktriangle$}--(1,-1)node{$\blacklozenge$}--cycle;
        \draw(0,0)node{$\bigstar$};
        \draw[dashed](-3/2,-3/2)--(3/2,3/2);
    \end{tikzpicture}
    \caption{The six orbits $O_1^{(1)},\ldots,O_6^{(1)}$ of the action of the group generated by the reflection about $y=x$ on the polytope corresponding to $\PP^1\times\PP^1$ with $k=1$.}\lb{P2figure}
\end{figure}

\begin{proof}[Proof of \eqref{item BS sym-intro}$\Leftrightarrow$\eqref{item alphaGone-intro}]
    If $P$ is Batyrev--Selivanova symmetric, $P^{\Aut P}=\{0\}$. Therefore $\lambda P^{\Aut P}\subset P$ for any $\lambda>0$. By Theorem \ref{Song}, $\alpha_{G(\Aut P)}=1$.

    If $P$ is not Batyrev--Selivanova symmetric, $P^H$ contains a non-zero point. Since $P$ is bounded, there is $\lambda_0>0$ such that $-\lambda P^H\nsubseteq P$ for $\lambda>\lambda_0$. Therefore $\alpha_G\leq\frac{\lambda_0}{1+\lambda_0}<1$.
\end{proof}

By Blum--Jonsson's reformulation of \cite[Corollary 7.16]{BJ20}
the $\delta$-invariant of Fujita--Odaka \cite{FO18}
 $\delta(-K_X)=1$ if and only if $\Bc(P)=0$, which is equivalent to the existence of K\"ahler--Einstein metrics \cite[Corollary 1.2]{WZ04}. By Rubinstein--Tian--Zhang, $\delta_k(-K_X)=1$ if and only if $\Bc_k(P)=0$  \cite[Corollary 7.1]{RTZ21}.

\begin{theorem}{\rm\cite[Corollary 7.16]{BJ20},\cite[Corollary 7.1]{RTZ21}}\label{BJ}
    Let $P$ be the polytope associated to the toric Fano manifold. Then
    $$
        \delta\left(-K_X\right)=\min_i\frac{1}{1+\left\langle\Bc\left(P\right),v_i\right\rangle},\qquad\delta_k\left(-K_X\right)=\min_i\frac{1}{1+\left\langle\Bc_k\left(P\right),v_i\right\rangle}.
    $$
\end{theorem}

\begin{corollary}
    Let $P$ be the polytope associated to the toric Fano manifold. Then $\delta(-K_X)=1$ if and only if $\Bc(P)=0$, and $\delta_k(-K_X)=1$ if and only if $\Bc_k(P)=0$.
\end{corollary}
\begin{proof}
    If $\Bc(P)=0$, by Theorem \ref{BJ}, $\delta(-K_X)=1$. If $\Bc(P)\neq0$, by completeness of the fan, we can find $v_j$ in the half space $\{\langle\,\cdot\,,\Bc(P)\rangle>0\}$. Then
    $\delta(-K_X)\leq(1+\langle\Bc(P),v_j\rangle)^{-1}<1$.

    The proof for $\delta_k$ and $\Bc_k$ is identical.
\end{proof}

\subsection{Relation to canonical metrics}


In this subsection we finish the proof of Theorem \ref{relation of symm notions on metric level}. Given Theorem \ref{relation of symm notions} and Lemma \ref{lattice-centrally-symm-reductive}, we are left with proving the following:
\begin{theorem}\label{SaitoThm}
    The existence of $k$-anticanonical balanced metric is equivalent to $\Bc_k(P)=0$, i.e., $\delta_k(-K_X)=1$.
\end{theorem}

Indeed, even for non-toric cases, it is known that $\delta_k<1$ implies non-existence, while $\delta_k>1$ implies existence \cite{RTZ21}. The difficulty here is the borderline case: when $X$ is toric, $\delta_k=1$ also implies existence of $k$-anticanonical balanced metric.

\begin{proof}
    It was shown by Saito--Takahashi \cite[Theorem 1.2]{ST} that the existence of a $k$-anticanonical balanced metric implies F-polystability, and Hashimoto \cite[Theorem 1]{Has} showed the converse. The equivalence of F-polystability and $\Bc_k(P)=0$ is due to Saito \cite[Theorem A]{Saito}. This completes the proof of Theorem \ref{SaitoThm} and hence of Theorem \ref{relation of symm notions on metric level}.
\end{proof}

\begin{remark}\label{dimupto6Rem}
For toric Fano manifolds of dimension up to 6, the existence of K\"ahler--Einstein metrics implies Batyrev--Selivanova symmetry \cite[Proposition 1.4]{NP11}. More precisely, we can formulate the following:
    Let $X$ be a toric Fano manifold with $\dim X\leq6$. The following are equivalent:
    \begin{enumerate}[$\rm 1)$]
        \item $X$ is Batyrev--Selivanova symmetric;\label{symmetric}
        \item $\delta_k(-K_X)=1$ for all $k$;\label{delta_k=1}
        \item $\delta_k(-K_X)=1$ for all sufficiently large $k$;\label{delta_k=1 k large}
        \item $\delta(-K_X)=1$, which is equivalent to the existence of K\"ahler--Einstein metrics.\label{delta=1}
    \end{enumerate}
\end{remark}
\begin{proof}
    We have shown in Theorem \ref{relation of symm notions} that $\ref{symmetric})\Rightarrow\ref{delta_k=1})\Rightarrow\ref{delta_k=1 k large})\Rightarrow\ref{delta=1})$. By Nill--Paffenholz's exhaustive computer search \cite{Obr07},\cite[Proposition 2.1]{NP11} among toric Fano manifolds of dimension up to 8 whose associated polytope has barycenter $0$, there are precisely three of them (one in dimension $7$ and two in dimension $8$) that are not Batyrev--Selivanova symmetric (indeed they even have non-zero quantized barycenter as we show in Appendix \ref{Nill--Paffenholz's counterexamples}).
\end{proof}

\subsection{\texorpdfstring{\textup{(\ref{item central sym})}$\;\not\Leftarrow\;$\textup{(\ref{item BS sym})}}{(i)<-(ii)}}
Take, for instance the polytope associated to $\PP^2$: 
its vertices are $(-1,2)$, $(-1,-1)$ and $(2,-1)$. 
Let
\begin{align*}
    a=\left(\begin{matrix}
        -1 & 1 \\
        -1 & 0
    \end{matrix} \right).
\end{align*}
Then $a:(-1,2)\mapsto(-1,-1); (-1,-1)\mapsto(2,-1); (2,-1)\mapsto(-1,2)$. 
Therefore $a$ maps $P$ to itself, i.e., $a\in\Aut P$. Since $a$ does not have non-zero fixed point, $P$ is Batyrev--Selivanova symmetric. 
However, $P\not=-P$ 
 (see Figure \ref{Fig for R(P)}).

\subsection{\texorpdfstring{\textup{(\ref{item BS sym})}$\;\not\Leftarrow\;$\textup{(\ref{item Bc_k=0 for all k})}}{(ii)<-(iii)}}
We do not currently know if there is a counterexample,
but if it exists it should be in dimension at least 9, based
on Remark \ref{dimupto6Rem} and the calculation in the Appendix.
By Theorem we would get a sequence of identities on $P$
given that all the coefficients in the asymptotic expansion for
$\delta_k$ must vanish. Perhaps this could be useful.

\subsection{\texorpdfstring{\textup{(\ref{item Bc_k=0 for large k})}$\,\not\Leftarrow\;$\textup{(\ref{item Bc=0})}}{(iv)<-(vi)}}
Nill--Paffenholz's counterexamples have $\Bc(P)=0$, but are not Batyrev--Selivanova symmetric. In fact, as we compute in 
Appendix \ref{Nill--Paffenholz's counterexamples}, they
satisfy $\Bc_k(P)\not=0$ for $k=1,2,3$,
that by Theorem
\ref{relation of symm notions}
means $\Bc_k(P)\not=0$
for all $k$ except for up to $n$ values $k$.
In addition, by \cite[Proposition 2.1]{NP11}, those counterexamples have the $1$-dimensional fixed space for $\Aut P$.
Since $\Bc_k(P)$ is fixed by $\Aut P$, then all $\Bc_k(P)$ lie an same line. Thus $\Bc(P)$ and $\Bc_k(P)$ lie on the same line.

\subsection{\texorpdfstring{\textup{(\ref{item Bc=0})}$\,\not\Leftarrow\;$\textup{(\ref{item R(P)=-R(P)})}}{(vi)<-(vii)}}
\begin{enumerate}[(1)]
    \item Futaki's 4-dimensional example: for many years after Matsushima's theorem, it was not known whether the reductivity of the automorphism group is merely a necessary condition for the existence of a K\"ahler--Einstein metric. Futaki showed that indeed is the case by introducing his invariant and computing the following example \cite[Section 3]{Fut83} and \cite[Remark 4.9]{Rub14}.
    
Consider the bundle $E:\cO_{\PP^1}(-1)\oplus\cO_{\PP^2}(-1)\rightarrow\PP^1\times\PP^2$.
Let $X$ be the total space of projective bundle $\PP(E)$ over $\PP^1\times\PP^2$. Since $\cO_{\PP^1}(-1)$ and $\cO_{\PP^2}(-1)$ are toric line bundles, then $\PP(E)$ is a toric manifold of dimension $4$. Let $P$ be its associated polytope.
By \cite[Section 3]{Fut83} and \cite{Wan91}, it does not admit K\"ahler--Einstein metrics, i.e., $\Bc(P)\neq 0$ (Remark \ref{existence of KE}). 
By \cite[Proposition 3.2]{Sak86}, we have that
\begin{align*}
    \Aut_0X\simeq& \Aut_0(\PP^1\times\PP^2)\times\CC^* \\
    =& \mathrm{PGL}(2, \CC)\times \mathrm{PGL}(3, \CC)\times\CC^*,
\end{align*}
which is reductive, i.e., by Lemma \ref{lattice-centrally-symm-reductive}, the associated polytope is centrally lattice symmetric.

Futaki's 4-dimensional example generalizes as follows. Let $X$ be the blow-up of $\PP^{n_1+n_2+1}$ along $V(Z_0,\ldots,Z_{n_1})$ and $V(Z_{n_1+1},\ldots,Z_{n_1+n_2+1})$, i.e.,
\begin{equation}\label{X}
    X=\{(u,v,w)\in\PP^n\times\PP^{n_1}\times\PP^{n_2}\,:\,(u_0,\ldots,u_{n_1})\in v, (u_{n_1+1},\ldots,u_{n_1+n_2+1})\in w\}.
\end{equation}
The total space of the bundle $E:\cO_{\PP^{n_1}}(-1)\oplus\cO_{\PP^{n-2}}(-1)\to\PP^{n_1}\times\PP^{n_2}$ is realized as
$$
\{(v,w,x,y)\in\PP^{n_1}\times\PP^{n_2}\times\CC^{n_1+1}\times\CC^{n_2+1}\,:\,x\in v, y\in w\}.
$$
After taking the projective space of $\CC^{n_1+1}\times\CC^{n_2+1}\cong\CC^{n_1+n_2+2}$, we obtain the total space of the projective bundle $\PP(E)$
$$
\left\{(v,w,u)\in\PP^{n_1}\times\PP^{n_2}\times\PP^{n_1+n_2+1}\middle|(u_0,\ldots,u_{n_1})\in v, (u_{n_1+1},\ldots,u_{n_1+n_2+1})\in w\right\},
$$
which is the same as \eqref{X}.

$X$ has a toric structure with the primitive elements of the rays of its fan being the columns of the matrix
$$
\begin{pmatrix}
    1   &   &       &   &-1     &0      &0\\
        &1  &       &   &-1     &\vdots &\vdots\\
        &   &       &   &       &0      &0\\
        &   &\ddots &   &\vdots &-1     &1\\
        &   &       &   &       &\vdots &\vdots\\
        &   &       &1  &-1     &-1     &1
\end{pmatrix},
$$
where the first $n_1$ rows of the last two cloumns are $0$.

Indeed, the first $n_1+n_2+2$ columns correspond to $\PP^{n_1+n_2+1}$ with each column corresponding to the hyperplanes $V(Z_1),\ldots,V(Z_{n_1+n_2+1}),V(Z_0)$, respectively. Recall \cite[(3.3.3)]{CLS11}. Since $V(Z_0,\ldots,Z_{n_1})$ corresponds to the cone spanned by $(-1,\ldots,-1)^\mathsf{T}$ and the first $n_1$ columns, blowing it up adds a new ray spanned by their sum, namely $(0,\ldots,0,-1,\ldots,-1)^\mathsf{T}$, where the first $n_1$ elements are $0$. Similarly, the cone corresponding to $V(Z_{n_1+1},\ldots,Z_{n_1+n_2+1})$ is spanned by the remaining $n_2+1$ columns, and blowing it up adds a new ray spanned by $(0,\ldots,0,1,\ldots,1)^\mathsf{T}$, where the first $n_1$ elements are $0$.

For any $n_1,n_2\geq1$, $P$ is centrally lattice symmetric; for $n_1\neq n_2$, $\Bc(P)\neq0$. In that case, $X$ has a reductive automorphism group but admits no K\"ahler--Einstein metric \cite[Lemma 2.1]{Fut91}.

\item Toric Fano 3-fold $No. 5.2$ $X$ is  K-unstable manifold  (see \cite[p.90 Table 3.1, Theorem 3.16]{ACC23}).  First, one takes the blow-up of $\PP^3$ in two disjoint lines, denoted by $X_0$, which is toric Fano 3-fold $No. 3.25$. Then $X$ is obtained by blow-up of $X_0$ in two curves contracted by the birational morphism to $\PP^3$ which are both contained in one exceptional surface.

The fan of $\PP^3$ is generated by $(1,0,0), (0,1,0), (0,0,1), (-1,-1,-1)$.
The toric blowup is corresponding to the \textit{star subdivision} (see \cite[Definition 3.3.17]{CLS11}).
Then the fan of $X_0$ is generated by (see \cite[Case (i)]{WW82})
\begin{align*}
    \begin{pmatrix}
    1 &0 &0 &1  &-1  &-1   \\
    0 &1 &0 &-1 &0   &0 \\
    0 &0 &1 &0  &-1  &0
\end{pmatrix}.
\end{align*}
As classified in \cite[Section 3]{WW82}, $X$ has a toric structure with the primitive elements of the rays of its fan being the columns of the matrix
\begin{align*}
    \begin{pmatrix}
    1 &0 &0 &0  &0  &0  &0  &-1   \\
    0 &1 &0 &-1 &-1 &0  &1  &0\\
    0 &0 &1 &1  &0  &-1 &-1 &-1
\end{pmatrix}.
\end{align*}
and $(-K_X)^3=36$.
Then, by
$L=-K_X$, its associated polytope is given by
\begin{align*}
    P=\left\{(x_1,x_2,x_3)\in\RR^3~\left|~ 
    \begin{array}{cc}
         -1\leq x_1\leq 1-x_3, & -1\leq x_2\leq 1, \\
         -1\leq x_3\leq 1, & -1\leq x_3-x_2\leq 1.
    \end{array}\right.
    \right\}.
\end{align*}

\begin{figure}
    \centering
    \begin{tikzpicture}
        \draw({-3/8},{-1-3/8*sqrt(3)})--({-3/8},{-3/8*sqrt(3)})--({1-3/8},{-3/8*sqrt(3)})--({1-3/8},{-1-3/8*sqrt(3)})--cycle--(-1,-1)--(-1,1)--({-3/8},{-3/8*sqrt(3)});
        \draw(-1,1)--({-1+3/8},{2+3/8*sqrt(3)})--({3/8},{2+3/8*sqrt(3)})--(1,1)--({1-3/8},{-3/8*sqrt(3)});
        \draw({1-3/8},{-1-3/8*sqrt(3)})--(1,-1)--(1,1);
        \draw[dashed](1,-1)--({3/8},{-1+3/8*sqrt(3)})--({-1+3/8},{-1+3/8*sqrt(3)})--(-1,-1);
        \draw[dashed]({3/8},{-1+3/8*sqrt(3)})--({3/8},{2+3/8*sqrt(3)});
        \draw[dashed]({-1+3/8},{-1+3/8*sqrt(3)})--({-1+3/8},{2+3/8*sqrt(3)});
        \draw[gray,->](0,1)--(0,3.5)node[left]{$x_1$};
        \draw[gray,dashed](1,0)--(0,0)--(0,1);
        \draw[gray,->](1,0)--(2,0)node[below]{$x_2$};
        \draw[gray,dashed](0,0)--({-3/8},{-3/8*sqrt(3)});
        \draw[gray,->]({-3/8},{-3/8*sqrt(3)})--(-1.2,{-1.2*sqrt(3)})node[left]{$x_3$};
        \filldraw(0,-1)circle(1pt);
        \filldraw(0,1)circle(1pt);
    \end{tikzpicture}
    \caption{The polytope $P$ for the toric Fano 3-fold $No. 5.2$. As marked in the graph, $R(P)=\{(-1,0,0),(1,0,0)\}$. In particular, $R(P)=-R(P)$.}
    \label{Fano 3fold No5.2}
\end{figure}
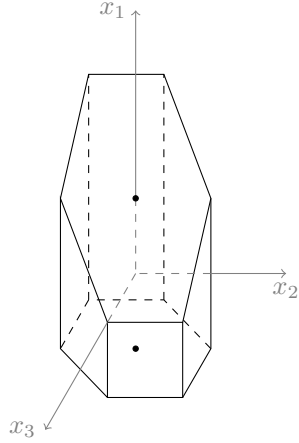

Letting, $\Bc(P)=(\Bc(P)_1,\Bc(P)_2,\Bc(P)_3)$, then
\begin{align*}
    &\Vol(P)\Bc(P)_1     \\
    =&\int_P x_1dx_1dx_2dx_3 \\
    =&\int_{x_2}\int_{x_3}\int_{-1}^{1-x_3}x_1dx_1dx_3dx_2 \\
    =&\int_0^1\int_{x_2-1}^1(\frac{1}{2}x_3^2-x_3)dx_3dx_2 +\int_{-1}^0\int_{-1}^{x_2+1}(\frac{1}{2}x_3^2-x_3)dx_3dx_2 \\
    =&\int_0^1(-\frac{1}{3}-\frac{1}{6}(x_2-1)^3+\frac{1}{2}(x_2-1)^2)dx_2 
    +\int_{-1}^0(\frac{1}{6}(x_2+1)^3-\frac{1}{2}(x_2+1)^2+\frac{2}{3})dx_2     \\
    =&-\frac{1}{8}+\frac{13}{24}   
    =\frac{5}{12}.
\end{align*}
Similarly, one has
\begin{align*}
    \Vol(P)\Bc(P)_2
    =&\int_P x_2dx_1dx_3dx_2 \\
    =&\int_{x_2}\int_{x_3}\int_{-1}^{1-x_3}x_2 dx_1dx_3dx_2 \\
    =&\int_0^1\int_{x_2-1}^1(x_2(2-x_3))dx_3dx_2 +\int_{-1}^0\int_{-1}^{x_2+1}(x_2(2-x_3))dx_3dx_2 \\
    =&\frac{9}{8}-\frac{37}{24}
    =-\frac{5}{12}
\end{align*}
and
\begin{align*}
    \Vol(P)\Bc(P)_3
     =&\int_P x_3dx_1dx_3dx_2 \\
    =&\int_{x_2}\int_{x_3}\int_{-1}^{1-x_3}x_3 dx_1dx_3dx_2 \\
    =&\int_0^1\int_{x_2-1}^1(x_3(2-x_3))dx_3dx_2 +\int_{-1}^0\int_{-1}^{x_2+1}(x_3(2-x_3))dx_3dx_2 \\
    =&\frac{1}{4}-\frac{13}{12}
    =-\frac{5}{6}.
\end{align*}
Since $\Vol(P)=\frac{1}{3!}(-K_X)^3=6$, then 
\begin{align*}
    \Bc(P)=(\frac{5}{72},-\frac{5}{72},-\frac{5}{36})\neq0,
\end{align*}
showing directly that $X$ does not admit a K\"ahler--Einstein
metric (and is K-unstable).

Its automorphism group is reductive with the connected component containing the identity being $\mathrm{GL}(2, \CC)\times\mathbb{G}_m$  \cite[p. 273, Table 6.1]{ACC23}, where $\mathbb{G}_m$ the multiplicative group $(\CC^*,\,\cdot\,)$.
Thus, its associated polytope $P$ satisfies 
$R(P)=-R(P)$ by 
Remark \ref{existence of KE} and
Lemma \ref{lattice-centrally-symm-reductive}.
In fact, we can compute $R(P)$ directly, see Figure \ref{Fano 3fold No5.2}.
In addition,
one has $\alpha_G(X)=1/2$.  
\end{enumerate}

\section{Demazure's Structure Theorem}\label{DemazureTheo}

\input{Notation_idea_and_examples}

\section{Proof of Theorem \ref{Demazure's structure theorem}}\label{Proof of Demazure's Theorem}

\input{Proof_of_Demazure_theorem}



\appendix
\section{Quantized barycenters of Nill--Paffenholz's examples}
\label{Nill--Paffenholz's counterexamples}


This Appendix is used in Remark \ref{dimupto6Rem}. Nill--Paffenholz's exhaustive computer search shows there are three such manifolds $X_1$, $X_2$, $X_3$ whose associated polytopes are non-symmetric but has barycenter $0$ \cite[Proposition 2.1]{NP11}.

$X_1$ is 7-dimensional. Let $P_1$ be the associated polytope. Then a computer calculation shows
\begin{align}
\label{P1BCEq}
\Bc_1(P_1)&=\frac{16}{2257}(-1,-1,-1,1,1,1,2), \nonumber \\
\Bc_2(P_1)&=\frac{60}{27121}(-1,-1,-1,1,1,1,2), \nonumber \\
\Bc_3(P_1)&=\frac{2744}{2579721}(-1,-1,-1,1,1,1,2).
\end{align}
It is interesting to note that our results give a refinement of
a theorem of Ono--Sano--Yotsutani \cite[Theorem 1.6]{OSY12}, who showed 
that $(X_1,-kK_{X_1})$ is not Chow semistable for $k$ large enough.
Chow stability of $(X_1,-kK_{X_1})$ is equivalent to existence of $k$-balanced metrics (see \cite[Theorem 4]{PS03}, \cite[Theorem 3.2]{Zh96}, \cite[Theorem 0.1]{Luo98}), so this means Ono--Sano--Yotsutani showed that $X_1$ does not admit $k$-balanced metrics for $k$ large enough. It was first shown by Saito--Takahashi \cite[Example 5.7]{ST} that $X_1$ does not admit $k$-anticanonical balanced metrics for large $k$. But in fact from the above computation and Theorem \ref{relation of symm notions}, $\delta_k(-K_{X_1})<1$ for all $k$ except at most $7$ values of $k$. Thus, $X_1$ does not admit $k$-anticanonical balanced metrics for all but at most $7$ values of $k$.

Next, $X_2=X_1\times\PP^1$. So $P_2=P_1\times[-1,1]$. Then $\Bc_k(P_2)=(\Bc_k(P_1),0)$ (see \eqref{P1BCEq}).

Finally,
$X_3$ is 8-dimensional. Let $P_3$ be the associated polytope. Then a computer calculation shows
\begin{align*}
\Bc_1(P_3)&=\frac{32}{5459}(-1,-1,-1,1,1,1,1,2),\\
\Bc_2(P_3)&=\frac{580}{321787}(-1,-1,-1,1,1,1,1,2).
\end{align*}

\bigskip
\textsc{
University of Maryland
and Princeton University
}

{\tt cjin123@terpmail.umd.edu}

\bigskip\textsc{
University of Maryland
and Stanford University
}

{\tt yanir@alum.mit.edu}

\bigskip
\textsc{Nanjing University
and Rutgers University}

{\tt zyang2000@qq.com}

\end{document}

%% file: introduction.tex
	The automorphism group of a toric variety $X$, denoted by $\Aut X$, consists of all algebraic automorphisms of $X$. A caveat is that such maps need not preserve the toric structure. Thus, the complete analysis of $\Aut X$ is { not} as simple (e.g., purely combinatorial) as one might initially expect.
	
	In 1970, Demazure \rm{\cite[Proposition 11, p. 581]{Dem70}} presented his structure theorem for the automorphism group of a smooth toric variety. Roughly, $\Aut X$ decomposes into subgroups related to the combinatorial data of the toric variety (i.e., its fan or polytope). 	
	Specifically, $\Aut_0 X$ (the identity component of $\Aut X$) contains the complex torus acting on $X$ as a maximal torus and a unipotent subgroup corresponding to what Demazure coined as the {\it roots} of $X$. These are denoted by $R(P)$ (Definition \ref{roots})
    since they can be identified with the lattice points described in Definition \ref{SymmDef} (c) using the same notation.

	In 1995, Cox provided a new point of view on Demazure's theorem
via the {\it homogeneous coordinate ring} (see \ref{HomoRing}). The advantage
of this ring is that it is intrinsic to $X$ and not associated to a particular projective embedding of $X$  \rm{\cite[\S 4]{Cox95}}. 
This also showed that Demazure's theorem extends verbatim to the orbifold
case, also referred to as the simplicial case by algebraic geometers
(Definition \ref{fanTermi}). 
Overall, the statement of Demazure's theorem is that, for an orbifold toric variety $X$, $\Aut X$ is generated by the maximal torus, certain {\it root automorphisms} associated with the lattice point in the relative interior of the $1$-dimensional faces of the polytope, and automorphisms induced by symmetries of the fan. The precise statement is:

    \begin{theorem}{\rm (Demazure's Structure Theorem)}
        Let $X=X_\Delta$ be a compact toric orbifold. Then the identity component $\Aut_0X$ of its automorphism group $\Aut X$ is a linear algebraic group (see Definition \ref{reductivegp}) satisfying:
    \begin{enumerate}[$(\rm i)$]
       \item $(\mbc^*)^n$ is a maximal algebraic torus in $\Aut_0X$. $\Aut_0X$ is the semi-direct product of its unipotent radical with a reductive group, which is generated by the maximal algebraic torus and the one-parameter groups corresponding to the semisimple roots $R_s(P)$.
       
        \item The unipotent radical of $\Aut_0X$ is generated by the one-parameter groups $\phi^{m,\lambda}$ for $m\in R_u(P)$ 
        (see \eqref{rootsAut} and Definition \ref{roots}).
    
    \item $\Aut \Delta$ can be naturally embedded in $\Aut X$ and $\Aut X$ is the product $\Aut \Delta \Aut_0X$, where $\Aut \Delta$ is the linear automorphism group of the corresponding fan $\Delta$ (see \eqref{AutDeltaEq}).
    
    \end{enumerate}
    \end{theorem}
    
    An even more precise version is stated in Theorem \ref{Demazure's structure theorem}. 
    \\

    \begin{remark}
        As Oda remarked (see the quote from his book provided above), it seems that Demazure's theorem should hold also for non-compact toric varieties, but it is not clear to us whether the original treatment states that case explicitly. It does seem that Demazure's article assumes $X$ is smooth. Cox's approach treats the orbifold setting. On the other hand, Demazure's result is not valid for non-compact case. For instance, the group $\aut(X)$ is already not a linear algebraic group for $X=\mbc^2$: Consider the subgroup $U$ of $\aut(\mbc^2)$ consisting of the maps $(z,w)\mapsto(z,w+f(z))$ where $f(z)\in\mbc[z]$. $U$ is isomorphic to $\mbc[z]$ (as an additive group), so $U$ is infinite-dimensional over $\mbc$. The reason Cox's proof deals with the orbifold case easily is that it lifts everything to an affine complex spaces by Theorem \ref{quotient-construction}, and this theorem is valid for simplicial toric varieties (or equivalently, toric orbifolds, see Theorem \ref{Property}).
    \end{remark}

    \begin{remark}
        Bruns--Gubeladze studied the automorphism group of any projective toric variety  \rm{\cite{BG99}}.
        Their results are in a sense parallel to Cox's:  contrary to Cox they require the variety to be projective,
        but they allow the variety to be more singular than orbifold. 
        The first point allows them to  use the information of a (very ample) polytope to describe the automorphism group. They use a different method to construct the toric variety, viewing a toric variety as $\proj S_P$, where $S_P$ is a graded ring determined by a fixed (very ample) polytope $P$ (actually $P$ determines a very ample line bundle on $X$, and $S_P$ is the corresponding projective coordinate ring). On the other hand, Cox uses only the information of the fan to describe the automorphism group, so his proof can treat non-projective varieties. 
        On the flip side, Cox requires the singularity of the variety to be at least simplicial to construct the variety by a geometric quotient (see Theorem \ref{quotient-construction}).
        In a more recent paper, it is claimed that Cox's results could be generalized to remove the simplicial condition \rm{\cite{MSS18}}, with the general idea the same as Cox's: showing the relation between the graded automorphism group of Cox's homogeneous coordinate ring $S$ and $\Aut_0 X$.
    \end{remark}

    \begin{remark}
        There are some interesting results on the class of complete toric varieties X such that a maximal unipotent subgroup $U$ in $\aut(X)$ acts on $X$ with an open orbit (so called \rm{radiant tori variety}). In \rm{\cite{AR17}} and \rm{\cite{APS24}}, the authors give a combinatorial criterion for a toric variety to be radiant and describe the maximal unipotent subgroup $U$ by the roots explicitly.
    \end{remark}

	

%% file: Notation_idea_and_examples.tex
\subsection{Notation and preparation}

	Given a complete toric variety $X$ ($X$ being complete means $X$ is proper, and is equivalent to the support of $\Delta$ filling the entire space $N_\RR$, see Definition \ref{Property}), we denote the Weil divisor class group of $X$ (or equivalently, degree $1$ part of Chow ring of $X$) by $\clx$:

    \begin{definition}\label{clx}
        $\mathrm{Div} X$ is the free abelian group generated by all the prime Weil divisors on $X$. The principal divisors form a subgroup $\mathrm{Div}_0 X$ of $\mathrm{Div} X$, and the \emph{Weil divisor class group} of $X$ is defined to be the quotient group
        \beq\label{clxeq}\clx:=\mathrm{Div} X/\mathrm{Div}_0 X.\eeq
    \end{definition}

    If $X$ is smooth, $\clx$ is isomorphic to the Picard group of $X$. But for simplicial cases, $\clx$ is different from the Picard group. To be specific, if $X$ is simplicial, then every Weil divisor has an integer multiple that is Cartier, so the Picard group is a subgroup of $\clx$ with finite index. For example, the divisor group of the weighted projective space $\mathbb{P}(1,1,2)$ is isomorphic to $\mathbb{Z}$, and the picard group corresponds to the subgroup $x\mathbb{Z}$. For readers interested in the details, we refer to Cox's textbook \cite[\S 4.0]{CLS11}.

	Denoting the set of all primitive vectors of $1$-dimensional cones in $\Delta$ by $\Delta_1$ \eqref{Delta1Eq}. Every vector $v\in\Delta_1$ corresponds to a prime divisor $D_v$ in $X$. We denote by 
    \[\mbc^{\Delta_1}:=\oplus_{v\in\Delta_1}\mbc v\] 
    the complex linear space spanned by the independent basis $\Delta_1$, and by \[(\mbc^*)^{\Delta_1}:=\oplus_{v\in\Delta_1}\mbc^* v\]
    the maximal torus in it.
    
    \begin{definition}\label{HomoRing}
        The \emph{homogeneous coordinate ring }
        \beq S:=\mbc[x_v]\eeq 
        of $X$ is a graded polynomial ring over $\mbc$ generated by the variables $\{x_v|v\in\Delta_1\}$. Every monomial $\prod x_v^{a_v}$ in $S$ corresponds to a divisor $D=\sum a_vD_v$, thus a divisor class in $\clx$. In this way, we  grade $S$ by $\clx$.
    \end{definition}
    
     For each $v\in\Delta_1$, the corresponding divisor is denoted by $D_v$. 

    We divide $\{x_v|v\in\Delta_1\}$ into \beq\label{DeltaAlphaEq}\bigcup_{\alpha\in\clx} \Delta_\alpha=\bigcup_{\alpha\in\clx}\{x_v\,:\,(D_v)\in\alpha\},\eeq
    and define $\mathcal{C}$ to be the set of the divisor classes that contains some $x_v$, namely \beq\label{LinearClassEq}
    \mathcal{C}:=\{\alpha\in\clx\,:\,\Delta_\alpha\neq\emptyset\}.
    \eeq
	
	Now applying the exact contravariance functor $\mathrm{Hom}_\mbz(\cdot, \mbc^\ast)$ to the exact sequence
	\[\xymatrix{0 \ar[r] & M \ar[r] & \bigoplus_v\mbz D_v \ar[r] & \clx \ar[r] & 0},\]
	we get an exact sequence: 
	\beq\label{G-ExactSeq}\xymatrix{1 \ar[r] & G \ar[r] & (\mbc^\ast)^{\Delta_1} \ar[r] & T'_X \ar[r] & 1},
    \eeq
	where 
    \beq\label{GEq} G:=\mathrm{Hom}_\mbz(\clx, \mbc^\ast)
    \eeq 
    and 
    \beq\label{Torusac}T'_X:= \mathrm{Hom}_\mbz(M, \mbc^\ast).\eeq 
    We will see that $T'_X$ corresponds to the action of the maximal torus in $X$, denoted by $T_X$, on $X$.
	
	For each cone $\sigma\in \Delta$, let $\sigma_1$ denote the set of all primitive generators of the dimension-$1$ faces of $\sigma$ (similar to \eqref{Delta1Eq}), and $x^{\hat{\sigma}}$ be the monomial
    \[x^{\hat{\sigma}}:=\prod_{v\in\Delta_1\setminus\sigma_1}x_v.\]
	We define an ideal $B\in S$ (the ``irrelevant ideal'')
	\[B:=\langle x^{\hat{\sigma}}|\sigma\in\Delta\rangle.\]
	As shown in the exact sequence above, $G$ acts on the affine space $\mbc^{\Delta_1}$ as a subgroup of $(\mbc^\ast)^{\Delta_1}$, leaving $V(B)$ (the subvariety defined by the ideal $B$) invariant.
	We have a quotient construction of $X$:
	\begin{theorem}[{\cite[\S 5.2]{CLS11}}]
    \label{quotient-construction} Let $X$ be a toric variety determined by a fan $\Delta$, $G$ be the group in \eqref{GEq}. We denote $\mbc^{\Delta_1}\setminus V(B)$ by $\tx$. Then:
    \begin{enumerate}[$(\rm i)$]
        \item $V(B)$ is invariant under the action of $G$.
        \item $X$ is the categorical quotient $\tx//G$.
        \item $X$ is simplicial (see Theorem \ref{Property}) if and only if it is the geometric quotient $\tx/G$.
    \end{enumerate}

	\end{theorem}
    \begin{remark}
        Roughly speaking, the categorical quotient means that every $G$ invariant morphism from $\tx/G$ to another scheme can descend to $X$. The geometric quotient means that there is a natural one-to-one correspondence between the points in $X$ and the $G$-orbits in $\tx$. For precise definitions, we refer to \rm{\cite[\S 5.0]{CLS11}}.
    \end{remark}

    We now show that $B$ is actually independent of the choice of generators of $S$. To be precise, we have the following.
    
    \begin{lemma}\label{IrreIdeal}
        If $X$ is complete and $\phi: S\to S$ is a degree-preserving automorphism of $S$, then $\phi(B)=B$.
    \end{lemma}

    \begin{proof}
        We need to prove that for any maximal cone $\sigma\in \Delta$, $\phi(x^{\hat{\sigma}})\in B$. By simple induction, it suffices to show that if $v_0\in \Delta_1\setminus \sigma_1$, and $\prod x_{v}^{a_v}$ is a monomial of the same degree as $x_{v_0}$, then $x^{\hat{\sigma}}\prod x_v^{a_v}/x_{v_0}\in B$.

        Now, since the divisor $\sum a_vD_v-D_{v_0}$ is a principal divisor, there is a character $m\in M$ such that $(m)=\sum a_vD_v-D_{v_0}$. This means $\langle m, v_0\rangle=-1$ and $a_v=\langle m, v\rangle\geqslant 0$ for $v\neq v_0$. Hence, the hyperplane $H$ perpendicular to $m$ divided $N\otimes\mathbb{R}$ into two parts, and one part contains only one vector in $\Delta_1$, namely $v_0$. By definition, the variable $x_v$ appears in $x^{\hat{\sigma}}\prod x_v^{a_v}/x_{v_0}$ if and only if $v\neq v_0$ and $v\not\in H\cap\sigma$. 
        
        Since $H\cap \sigma$ is a face of $\sigma$ and $\Delta$ is complete (see Definition \ref{Property}), if we choose a ray $l$ in the relative interior of $H\cap \sigma$ (if $H\cap \sigma$ is a ray we just choose itself), then for any point $p$ in the relative interior of the cone spaned by the ray generated by $v_0$ and $l$, there is a cone $\sigma'\in\Delta$ contains $p$. But the ray generated by $v_0$ is the only dimension-$1$ cone of $\Delta$ on the same side of $x$, $\sigma'$ must contain $v_0$ as a face. For the same reason, $\sigma'$ must contain $l$. By the definition of fan, $\sigma'\cap H\cap \sigma$ is a face of both $H\cap \sigma$ and $\sigma'$, so $\sigma'\cap H\cap \sigma=H\cap \sigma$. Hence, $H\cap \sigma$ is a face of $\sigma'$. Thus, $x^{\hat{\sigma'}}$ is a factor of $x^{\hat{\sigma}}\prod x_v^{a_v}/x_{v_0}$, which implies what we want.
    \end{proof}   

    
    \begin{remark}\label{cod2argum}
        Since for every varieble $x_v$ (see Definition \ref{HomoRing}), there exists a $\sigma\in\Delta$ such that the monomial $x^{\hat{\sigma}}$ does not contain $x_v$, $V(B)$ has codimension at least $2$. This fact will be used several times in the proof below.
    \end{remark}

        \subsection{Statement and key idea of proof}
    
	The main task of the next section is to analyze the structure of the automorphism group $\Aut X$ of $X$. We will eventually establish the following results.
	\begin{theorem}[Demazure Theorem]\label{Demazure's structure theorem}
		Let $X$ be a complete simplicial toric variety of dimension $n$ (see Theorem \ref{Property}), $\Delta$ the corresponding fan and $P$ the corresponding polytope. Then its automorphism group $\Aut X$ is a Lie group that satisfies:
	\begin{enumerate}[$(\rm i)$]
		\item The connected component of $\Aut X$ that contains the identity, denoted by $\Aut_0 X$, is a linear algebraic group (see Definition \ref{reductivegp}) with $T'_X\cong (\mbc^*)^n$ (see \eqref{Torusac}) as a maximal torus. 
		
		\item  $\Aut_0X$ is the semi-direct product of its unipotent radical (see Definition \ref{reductivegp}) $R_u$ with a group $G_s/G$, where $G_s$ is a reductive group and $G$ is given by \eqref{GEq}. Furthermore, $R_u$ is generated by some one-parameter groups $\phi^{m,\lambda}$ \eqref{rootsAut} associated with the roots $m\in R_u(P)$ (Definition \ref{roots}), and $G_s\cong \prod_{\alpha\in\mathcal{C}}\mathrm{GL}(|\Delta_\alpha|,\mbc)$ \eqref{DeltaAlphaEq} contains $(\mbc^*)^{\Delta_1}$ as a maximal torus (namely, it is the group consisting of all diagonal matrices). The exact sequence \eqref{G-ExactSeq} shows how $G$ is embedded in $G_s$.
		
		\item  $\Aut \Delta$ (see \eqref{AutDeltaEq}) can be naturally embedded in $\Aut X$. Then we have $\Aut X/\Aut_0X\cong \Aut \Delta/(\Aut\Delta\cap \Aut_0 X)$. Furthermore, if $X$ is smooth and Fano, then $\Aut X/\Aut_0X\cong \Aut P/\Aut_0 P$ (see Definition \ref{Aut0P}) and $\Aut_0 P$ is isomorphic to the Weyl group of $G_s/G$ with respect to the maximal torus $T'_X$ (see Definition \ref{Weylgp}).
        \end{enumerate}
	\end{theorem}
    
	 As a corollary, we will get Lemma \ref{lattice-centrally-symm-reductive}.

     \begin{proof}[Proof of Lemma \ref{lattice-centrally-symm-reductive}]\label{SymEqualRed}
     From part (ii) of Theorem \ref{Demazure's structure theorem}, the unipotent radical of $\Aut_0X$ is generated by some one-parameter groups associated with the roots in $R_u(P)$. But $R(P)=-R(P)$ implies $R_u(P)=\emptyset$, so the unipotent radical is trivial, and $\Aut_0X$ is reductive.
         
     \end{proof}
    
	

	We now try to explain the basic idea of Cox's proof of Theorem \ref{Demazure's structure theorem}:

    \begin{enumerate}
        \item Since $\mbc^{\Delta_1}\setminus V(B)$ is simply an affine space minus some subvarieties cut out by coordinate functions, it is much easier to analyze its automorphism group. Therefore, we concentrate on the group of automorphisms of $\mbc^{\Delta_1}\setminus V(B)$ that can be descended to $X$, whose identity component turns out to consist of automorphisms formed by some homogeneous polynomials with respect to the $\clx$-grading. 
        \item Observing that these homogeneous polynomials are determined by the degree-preserving linear endomorphism of $S$, we can write them as matrices. This shows that they form a linear algebraic group and allows us to figure out the precise structure of this group.
        \item To understand how the automorphism group $\Aut X$ is related to the combinatorial information of $X$, we study the relation between its unipotent radical $R_u$ and the roots $R(P)$ of $X$.
        \item The next step is to describe the relation of the groups described above and $\Aut X$. Just as our intuition tells us, $\Aut X$ is the quotient group $\widetilde{\Aut} X/G$, but the proof requires effort. We first lift the automorphisms in the indentity component of $\Aut X$ to $\widetilde{\Aut}_0 X$ (see Definition \ref{aut0}).
        \item Finally, we show how the elements in $\Aut P$ "generate" other components of $\Aut X$.
    \end{enumerate}

    \begin{remark}
    \lb{simplifyremark}
            In the next section, we give a proof of Demazure's theorem following the paper of Cox. We modify the proof from Cox's original paper slightly to avoid some arguments.        
            We do not introduce the isomorphism between $\widetilde{\Aut}_0 X$ and the degree-preserving automorphism group of the homogeneous coordinate ring $S$. Instead, we simply view this as a representation over a linear subspace of $S$. However, the essential argument used in proving the isomorphism is already shown in the proof of Lemma \ref{IrreIdeal}.
            Some other places are simplified slightly. For example, the connectedness argument is used in the proof of Proposition \ref{AutClx} and the following corollary to replace some arguments in the original proof.

            Cox's original proof has a gap: he assumed that the automorphism group of a ring is an algebra, which is wrong, so that he can run some arguments more easily. The gap is, of course, a little serious, but it does not affect his general idea at all. And also it was pointed out very clearly and fixed very well in his erratum. The idea we explained above is the fixed one.
    \end{remark}

\subsection{Examples of toric Del Pezzos}

    Let us look at the examples of toric Del Pezzos. They give us some intuition on how to prove the theorem.

    \begin{table}[ht]
\centering
    
        \begin{tabular}{|c|c|c|c|}
    \hline
        $X$ & $\PP^2$ & $\PP^1\times\PP^1$ &  $\PP^2$ blown up $3$ points \\
        \hline
        $S$ & $\mbc[x_0,x_1,x_2]$ & $\mbc[x_0,x_1,y_0,y_1]$ & $\mbc[x_0,x_1,x_2,y_1,y_2,y_3]$ \\
        \hline
        $G$ & $\mbc^*$ & $(\mbc^*)^2$ &  $(\mbc^*)^4$ \\
        \hline
        $\tx$ & $\mbc^3\setminus\{0\}$ & $\mbc^4\setminus V(x_0,x_1)\cup V(y_0,y_1)$ & \makecell{$\mbc^3\setminus V((x_1y_1x_2y_2)\cap V(y_1x_2y_2x_3)$ \\ $\cap V(x_2y_2x_3y_3))\cap V(y_2x_3y_3x_1)$ \\ $\cap V(x_3y_3x_1y_1)\cap V(y_3x_1y_1x_2)$} \\
        \hline
        $\Aut X$ & $\mathrm{PGL}(3,\mbc)$ & \makecell{$(\mathrm{PGL}(2,\mbc)\times\mathrm{PGL}(2,\mbc))$\\$\rtimes\ZZ/2\ZZ$}  & $(\mbc^*)^2\rtimes D_{12}$ \\
         \hline
        $\Aut_0X$ & $\mathrm{PGL}(3,\mbc)$ & $\mathrm{PGL}(2,\mbc)\times\mathrm{PGL}(2,\mbc)$ & $(\mbc^*)^2$ \\
         \hline
        $\Aut P$ & $S_3$ & $D_8$ & $D_{12}$ \\
         \hline
        $\Aut_0 P$ & $S_3$ & $\ZZ/2\ZZ\times\ZZ/2\ZZ$ & $0$ \\
         \hline
        $R_u$ & $0$ & $0$ & $0$ \\
         \hline
        $G_s$ & $\mathrm{GL}(3,\mbc)$ & $\mathrm{GL}(2,\mbc)\times\mathrm{GL}(2,\mbc)$ & $(\mbc^*)^6$ \\
         \hline

    \end{tabular}
 \caption{\small K-polystable Toric Del Pezzo surfaces. In these cases $R_s(P)=R(P)$ and $R_u(P)=\emptyset$, so $\Aut_0 X$ is reductive. See Figure \ref{Fig for R(P)}. Here $V(f_1,...,f_i)$ means the subvariety defined by the ideal generated by $f_1,..., f_i$.}
\label{polyexamp}
\end{table}

\begin{table}[ht]
    \centering

        \begin{tabular}{|c|c|c|}
    \hline
        $X$ & $\PP^2$ blown up $1$ point & $\PP^2$ blown up $2$ points \\
        \hline
        $S$  & $\mbc[x_0,x_1',x_2',y]$ & $\mbc[x_0',x_1',x_2',y_1,y_2]$ \\
        \hline
        $G$ & $(\mbc^*)^2$ & $(\mbc^*)^3$  \\
        \hline
        $\tx$ & $\mbc^4\setminus V(x_0,y)\cup V(x_1',x_2')$ & \makecell{$\mbc^3\setminus V(x'_1,y_2)\cup V(x'_2,x'_0)$ \\$\cup V(y_2,y_1)\cup V(x'_0,x'_1)$ \\ $\cup V(y_1,x'_2)$} \\
        \hline
        $\Aut X$ & $\mbc^2\rtimes\mathrm{GL}(2,\mbc)$ & $(\mbc^2\rtimes(\mbc^*2))\rtimes\ZZ/2\ZZ$ \\
         \hline
        $\Aut_0X$ & $\mbc^2\rtimes\mathrm{GL}(2,\mbc)$ & $\mbc^2\rtimes(\mbc^*)^2$ \\
         \hline
        $\Aut P$ & $\ZZ/2\ZZ$ & $\ZZ/2\ZZ$ \\
         \hline
        $\Aut_0 P$ & $\ZZ/2\ZZ$ & $0$ \\
         \hline
        $R_u$ & $\mbc^2$ & $\mbc^2$ \\
         \hline
        $G_s$ & $\mathrm{GL}(3,\mbc)\times(\mbc^*)^2$ & $(\mbc^*)^5$ \\
         \hline
        $R_s(P)$ & $\{(1,-1),(-1,1)\}$ & $\emptyset$ \\
        \hline
        $R_u(P)$ & $\{(1,0),(0,1)\}$ & $\{(-1,0),(0,-1)\}$ \\
        \hline

    \end{tabular}
    \caption{\small K-unstable Toric Del Pezzo surfaces. See Figure \ref{Figure of blowup 1,2 points}. Here $V(f_1,...,f_i)$ means the subvariety defined by the ideal generated by $f_1,..., f_i$.}
\label{usexamp}
\end{table}

	\begin{example}
		The simplest but nontrivial example may be $\PP^2$. The homogeneous coordinate ring of it is $\mbc[x_0,x_1,x_2]$, and the automorphism group is $\mathrm{PGL}(3,\mbc)$, which is known to be reductive. According to Figure \ref{Fig for R(P)}, $\Aut P=S_3$, and the roots of $\PP^2$ is
		\[\{(0,1),(1,1),(1,0),(0,-1),(-1,-1),(-1,0) \},\]
		so the polytope of $\PP^2$ is centrally lattice symmetric, although the polytope itself is not centrally symmetric. 
		
		Easily seen from the Lie algebra of $\mathrm{PGL}(3,\mbc)$, $\mathrm{PGL}(3,\mbc)$ is generated by the maximal torus and the family below:
		\[\begin{pmatrix}1 & \lambda & 0 \\ 0 & 1 & 0 \\ 0 & 0 & 1\end{pmatrix},
		\begin{pmatrix}
			1 & 0 & \lambda \\ 0 & 1 & 0 \\ 0 & 0 & 1
		\end{pmatrix}, 
		\begin{pmatrix}
			1 & 0 & 0 \\ \lambda & 1 & 0 \\ 0 & 0 & 1
		\end{pmatrix}, 
		\begin{pmatrix}
			1 & 0 & 0 \\ 0 & 1 & \lambda \\ 0 & 0 & 1
		\end{pmatrix}, 
		\begin{pmatrix}
			1 & 0 & 0 \\ 0 & 1 & 0 \\ \lambda & 0 & 1
		\end{pmatrix}, 
		\begin{pmatrix}
			1 & 0 & 0 \\ 0 & 1 & 0 \\ 0 & \lambda & 1
		\end{pmatrix}. \]
		We will reveal the relation between them and the six roots of $\PP^2$ in the next section.
	\end{example}

	\begin{example}\label{ruled surface}
		To see how the $\clx$-grading works, we blow up a $0$-dimensional toric orbit, namely the intersection point $x$ of $D_1, D_2$, in $\PP^2$. Then we get an ruled surface $X$, the homogeneous coordinate ring of which is $\mbc[x_0,x'_1,x'_2,y]$. Here, $y$ corresponds to the exceptional divisor $E$, and $x'_i$ corresponds to the strict transform $D'_i$ of $D_i$. As a ruled surface, the divisor class group of $X$ is $\mbz\oplus\mbz$, where the first $\mbz$ is the classes of fibers ($D'_i$ are here), the latter $\mbz$ is the classes of sections ($E$ is here), and $D_0$ is contained in the divisor class $(1,1)$. 
		
		The automorphism group $\Aut X$ can be viewed a subgroup of $\mathrm{PGL}(3,\mbc)$, generated by the maximal torus and the family below:
		\[\begin{pmatrix}1 & \lambda & 0 \\ 0 & 1 & 0 \\ 0 & 0 & 1\end{pmatrix},
		\begin{pmatrix}
			1 & 0 & \lambda \\ 0 & 1 & 0 \\ 0 & 0 & 1
		\end{pmatrix}, 
		\begin{pmatrix}
			1 & 0 & 0 \\ 0 & 1 & \lambda \\ 0 & 0 & 1
		\end{pmatrix}, 
		\begin{pmatrix}
			1 & 0 & 0 \\ 0 & 1 & 0 \\ 0 & \lambda & 1
		\end{pmatrix}. \]
		This can be derived from the fact that the holomorphic vector fields on $X$ (in other words, a vector in the Lie algebra of $\Aut X$) correspond one to one to the holomorphic vector fields on $\PP^2$ that vanish at $x$. Equivalently, we can say that $\Aut X$ contains the elements in $\mathrm{PGL}(3,\mbc)$ that preserve the point $x=[1,0,0]$ (very roughly, this is quite reasonable since $E$ itself is the only prime divisor on $X$ that is linear equivalent to $E$, which forces every automorphism of $X$ preserves $E$). 
		
		Therefore, $\Aut X$ is not reductive. The first two one-parameter groups above generate the unipotent radical of $\Aut X$. The latter two one-parameter groups and the maximal torus in some sense generate the "reductive part" of $\Aut X$, which is isomorphic to $\mathrm{GL}(2,\mbc)$. 

        According to Figure \ref{Fig for R(P)}, $\Aut P=\Aut_0 P=\ZZ/2\ZZ$, which again shows that $\Aut X$ is connected in this case.

        See Remark \ref{ruled surface2} for further discussion.
	\end{example}
	
	\begin{example}\label{square}
		We briefly introduce another simple example, $\PP^1\times\PP^1$. In this case, the automorphism group is not connected, since, besides the identity component $\mathrm{PGL}(2,\mbc)\times\mathrm{PGL}(2,\mbc)$ (it is a reductive group generated by the maximal torus and four one-parameter groups corresponding to the roots), there is another special automorphism $P$ that exchanges the two $\PP^1$. This automorphism has a distinctive feature that makes it so different from the automorphisms appearing above: it induces a nontrivial map on the divisor class group. 

        We can see this phenomenon in $\Aut P$ directly: Here, $\Aut P$ is the dihedral group $D_8$, and $\Aut_0 P\cong \ZZ/2\ZZ\times\ZZ/2\ZZ$ is the subgroup generated by the up-down flip and the left-right flip. The quotient group is isomorphic to $\mbz/2\mbz$. Thus, $\Aut X\cong (\mathrm{PGL}(2,\mbc)\times\mathrm{PGL}(2,\mbc))\rtimes\ZZ/2\ZZ$. 

        In other toric Del Pezzos, $\Aut_0 P$ is equal to either the trivial group or $\Aut P$ itself. 
        \end{example}

        For other toric Del Pezzos, the automorphism group is listed here. 
        
        \begin{example}\label{pentagon}
        Suppose $X$ is $\PP^2$ blown-up two points. Then its homogeneous coordinate ring is $\mbc[x_0',x_1',x_2',y_1,y_2]$, where $y_i$ corresponds to the exceptional divisors and $x'_i$ are the strict transforms of the coordinate lines in $\PP^2$. 
        
        Similarly to Example \ref{ruled surface}, the identity component $\aut_0 X$ of the automorphism group of $\PP^2$ blown-up two points $[1,0,0],[0,1,0]$ is generated by the maximal torus $(\mbc^*)^2$ and the family: 
        \[\begin{pmatrix}
			1 & 0 & \lambda \\ 0 & 1 & 0 \\ 0 & 0 & 1
		\end{pmatrix}, 
		\begin{pmatrix}
			1 & 0 & 0 \\ 0 & 1 & \lambda \\ 0 & 0 & 1
	\end{pmatrix},\]
        since it must preserve the two points $[1,0,0],[0,1,0]$.

        The difference is that there is another component, the element of which exchanges the points $[1,0,0],[0,1,0]$. We choose a typical member of this:
        \[\begin{pmatrix}
			0 & 1 & 0 \\ 1 & 0 & 0 \\ 0 & 0 & 1
	\end{pmatrix}.\]
        This matrix generates a subgroup $\ZZ/2\ZZ$. This corresponds to $\Aut P$ according to Figure \ref{Figure of blowup 1,2 points}. Therefore, the automorphism group is a semi-direct product of $\ZZ/2\ZZ$ and $\aut_0X$.

        \end{example}
        
        \begin{example}
        The automorphism group of $\PP^2$ with three non-linear points blown up is isomorphic to the semi-direct product of $(\mbc^*)^2$ and the dihedral group $D_{12}$. Roughly speaking, the reason is as follows:

        The three exceptional divisors and the strict transforms of the lines connecting the three points form a regular hexagon: the vertices of the hexagon correspond to these six divisors, and there is an edge connecting two vertex if and only if the divisors corresponding to them intersect. The automorphism of $\PP^2$ with three nonlinear points blown up must preserve the set of these six divisors and their intersection relation. The subgroup of the automorphism group of $\PP^2$ that preserves the three blown up points is the torus $(\mbc^*)^2$ itself. The above discussion tells us that $\Aut X/(\mbc^*)^2$ can be identified as the automorphism group of the hexagon.

        In fact, there is a natural subgroup $D_{12}\cong \Aut P\sse\Aut X$ that contains the automorphisms induced by the isomorphisms of the polytope (see Figure \ref{Fig for R(P)}).

        \end{example}

	Keeping these examples in mind, we will figure out the structure of the automorphism groups of smooth toric varieties in the next section.

%% file: Proof_of_Demazure_theorem.tex
	We go through the proof step by step. 
    
\subsection{Step \texorpdfstring{$1.$}{1.} Lifting automorphisms to \texorpdfstring{$\tx$}{tx}}

	According to the basic idea we mention in the last section, if $X$ is simplicial, it can be viewed as a geometric quotient of $\tx$ by $G=\mathrm{Hom}_\mbz(\clx, \mbc^\ast)$ (see Theorem \ref{quotient-construction}), and we shall first find the automorphisms of $\tx$ that can descend to $X$. Such automorphisms $\phi$ must send one $G$-orbit to another. To be specific, for any $x\in \tx$ and $g\in G$, $\phi$ should satisfy
	\[\phi(g\cdot x)\in G\cdot\phi(x), \]
    which means \[\phi^{-1} G\phi=G.\] 
	
	Since $V(B)$ has codimension at least two (Remark \ref{cod2argum}), every $\phi\in\Aut \tx$ can be extended to an endomorphism of $\mbc^{\Delta_1}$ by the Riemann extension theorem. Thus, an automorphism by the Zariski-Nagata purity theorem, which asserts that on a regular algebraic variety the branch locus of a morphism must consist of pure codimension-$1$ subvarieties (see \cite{Za58}), we get $\Aut \tx\sse\Aut \mbc^{\Delta_1}$. So we define:
	
	\begin{definition}
		$\widetilde{\Aut} X$ is the normalizer of $G$ in $\Aut \tx$.
	\end{definition}

    As shown above, $\widetilde{\Aut} X$ is exactly the group containing all the automorphisms of $\tx$ that can descend to $X$.

	Every element in $G$ descends to the identity of $X$ since it preserves all the $G$-orbits. In fact, every $\phi\in \widetilde{\Aut} X$ that descends to the identity of $X$ is contained in $G$:
	
	If $\phi\in \widetilde{\Aut} X$ descends to the identity of $X$, it preserves every $G$-orbit. Write each component of $\phi$ as $\phi_v$, we find that the polynomial $\phi_v$ cannot contain any variables except $x_v$, otherwise it would send some point outside the coordinate hyperplane $V(x_v)$ to $V(x_v)$, contradicting the fact that all coordinate hyperplanes are invariant under $G$. Furthermore, $\phi_v$ must be a non-zero linear function because different points cannot be sent to one point. Thus, $\phi\in(\mbc)^{\Delta_1}$. Now, the exact sequence \eqref{G-ExactSeq} \[\xymatrix{1 \ar[r] & G \ar[r] & (\mbc^\ast)^{\Delta_1} \ar[r] & T'_X \ar[r] & 1}\] and the assumption that $\phi$ descend to the identity of $X$ show $\phi\in G$.
	
	The proof of the fact that every automorphism of $X$ can be lifted to $\widetilde{\Aut} X$ is more difficult and requires more preparation, so we skip it temporarily and put it at the end of this section (see Proof \ref{proofKey}). Having admitted this, we have:
	
	\begin{proposition}\label{KeyExact}
		$\xymatrix{1 \ar[r] & G \ar[r] & \widetilde{\Aut} X \ar[r] & \Aut X \ar[r] & 1}$ is an exact sequence.
	\end{proposition}

    \begin{remark}
        When $X=\PP^n$ this exact sequence is quite classical: 
    \[\xymatrix{1 \ar[r] & \mbc^* \ar[r] & \mathrm{GL}(n+1,\mbc) \ar[r] & \mathrm{PGL}(n+1,\mbc) \ar[r] & 1}.\]

    When $X=\PP^1\times\PP^1$, the exact sequence is
    \[\xymatrix{1 \ar[r] & (\mbc^*)^2 \ar[r] & \mbz/2\mbz\times(\mathrm{GL}(2,\mbc))^2 \ar[r] & \mbz/2\mbz\times(\mathrm{PGL}(n+1,\mbc))^2 \ar[r] & 1}.\] 
    The group $G=(\mbc^*)^2$ is mapped into $(\mathrm{GL}(2,\mbc))^2$ and makes the quotient group $(\mathrm{PGL}(n+1,\mbc))^2$. The $\mbz/2\mbz$ in two groups both correspond to the automorphism that exchanges the two $\PP^1$, thus canceling each other.
    \end{remark}

	The next step is to understand the structure of $\widetilde{\Aut} X$. It is natural to consider a smaller group formed by the automorphisms that commutes with all the elements of $G$, since they look simply as homogeneous polynomials with respect to the $\clx$-grading.
	
	\begin{definition}\label{aut0}
		$\widetilde{\Aut}_0 X$ is the centralizer of $G$ in $\Aut\tx$.
	\end{definition}

\subsection{Step \texorpdfstring{$2.$}{2.} Understanding the structure of \texorpdfstring{$\widetilde{\Aut}_0 X$}{waut0(X)}}
    
	In order to show that $\widetilde{\Aut}_0 X$ is a linear algebraic group (see Definition \ref{reductivegp}) over $\mbc$ and to analyze its structure, we construct a faithful representation (see Definition \ref{reductivegp}) of it.

    For $\phi\in\widetilde{\Aut}_0 X$, write each component $\phi_v$ as
	\[\phi_v=\sum c_{v,D} x^D, \]
	where $v\in\Delta_1$, $D$ is a Weil divisor $D=\sum a_v D_v$ and 
\beq\label{xd} x^D:=\prod x_v^{a_v}.\eeq Then for every $g\in G=\mathrm{Hom}_\mbz(\clx, \mbc^\ast)$
	\begin{equation*}
	    (\phi \circ g)_v=\sum g(D)c_{v,D} x^D, \quad (g\circ\phi)_v=g(D_v)\sum c_{v,D} x^D.
	\end{equation*}
	So, by the definition of $G$, every monomial $x^D$ appearing in $\phi_v$ must satisfy
    \[ g(D)=g(D_v) \]
    for every $g\in G$, thus having the same degree as $x_v$. Therefore, we have:

    \begin{definition}\label{representation}
        For each automorphism $\phi=(\phi_v)_{v\in\Delta_1}$ in $\widetilde{\Aut}_0 X$, the map $x_v\mapsto \phi_v$ gives a degree-preserving $\mbc$-algebra endomorphism of $S$. (Formally speaking, $\phi$ is an automorphism of the affine variety $\mbc^{\Delta_1}\cong \mathbf{Spec} S$, so it corresponds to an automorphism of $S$ naturally. The above discussion tells us that this automorphism preserves the $\clx$-grading structure of $S$.) 
        
        Restricting this endomorphism to $S_\alpha$ for $\alpha\in \mathcal{C}$ \eqref{LinearClassEq} ($S_\alpha$ means the degree $\alpha$ part of the homogeneous coordinate ring $S$), we represent $\widetilde{\Aut}_0 X$ over the linear space $\oplus_{\alpha\in\mathcal{C}} S_\alpha$.

        To be explicit, this representation
        \beq\label{representationEq} \widetilde{\Aut}_0 X\to \mathrm{GL}(\oplus_{\alpha\in\mathcal{C}} S_\alpha): \phi\mapsto\Phi_\phi \eeq
        sends $\phi\in \widetilde{\Aut}_0 X$ to a block-diagonal matrix $\Phi_\phi$, whose diagonal blocks are denoted by
        \[\Phi_\alpha\in \mathrm{GL}(S_\alpha).\]

    Every row (and column) of $\Phi_\phi$ corresponds to a monomial in $S$, thus a divisor $D$ of the form $\sum a_v D_v$ for $v\in\Delta_1$. So, a component of $\Phi_\phi$ is positioned by two such divisors $(D,D')$.
        
	For $D=\sum a_v D_v$ and $D'=\sum a'_v D_v$ in the same divisor class $\alpha$, the $(D,D')$-component $\varphi_{D,D'}$ of the block $\Phi_\alpha$ is the coefficient of term $x^{D'}$ (see \eqref{xd}) in the polynomial 
    \beq\label{phiDEq}\phi^D:=\prod \phi_v^{a_v}.\eeq
    Observing that the normal $\mbz$-degree of the nonzero term in $\phi^D$ can not be lower than that of $x^D$, under the natural monomial basis $\{x^D\,:\, \deg(D)=\alpha\}$ of $S_\alpha$ (partially ordered by the usual $\mbz$-grading), the block $\Phi_\alpha$ looks like
	\beq\label{matrixform}\Phi_\alpha=\begin{pmatrix}
		A_\alpha & B_\alpha \\ 0 & C_\alpha
	\end{pmatrix}, \eeq
	where $A_\alpha$ and $B_\alpha$ correspond to coefficients of $\phi_v$ for $v\in\Delta_\alpha$ \eqref{DeltaAlphaEq} and the order of $A_\alpha$ is the cardinality $|\Delta_\alpha|$.

    \end{definition}

	By definition, it is a faithful representation. By Lemma \ref{IrreIdeal}, every matrix in $\mathrm{GL}(\oplus_{\alpha\in\mathcal{C}} S_\alpha)$ of this form comes from an automorphism of $\mbc^{\Delta_1}$ that preserves $V(B)$, therefore coming from an automorphism in $\widetilde{\Aut}_0 X$. The image of this representation is a linear algebraic group, since it is a closed subvariety defined by the following equations in $\mathrm{GL}(\oplus_{\alpha\in\mathcal{C}} S_\alpha)$:
	
	$\varphi_{D,D'}=0$ for those components not in $A_\alpha, B_\alpha, C_\alpha$; for components $\varphi_{D,D'}$ in $C_\alpha$, the coefficient of term $x^{D'}$ in $\prod \phi_v^{a_v}$ is the polynomial of some components in $A_\alpha, B_\alpha$, and the equation follows.
	\\
    
	{\bfseries From now on we identify $\widetilde{\Aut}_0 X$ with its image under the representation \eqref{representationEq}.}
    \\
	
	Noting that the components of all $A_\alpha, B_\alpha$ are independent of each other and the components in $C_\alpha$ are determined by them, we know that
	\[\label{dimension}\dim\widetilde{\Aut}_0 X=\sum|\Delta_\alpha|\dim S_\alpha. \]
	(To make sure $\Phi_\phi$ is invertible, $A_\alpha, C_\alpha$ must be invertible, but it is an open condition (i.e., stable under small perturbation) and does not affect the dimension of $\widetilde{\Aut}_0 X$.)
	\\
    
	Now we can find the unipotent radical (see Definition \ref{reductivegp}) of $\widetilde{\Aut}_0 X$. Take an automorphism $\phi$ and write $\Phi_\phi$ as \eqref{matrixform}. Suppose $B_\alpha=0$ for all $\alpha$, then $\Phi_\phi$ is totally determined by $A_\alpha$. Furthermore, $\Phi_\phi$ of this type has an inverse of the same type because in this case $\phi$ is just an invertible linear automorphism over $\mbc^{\Delta_1}$. So, all $\Phi_\phi$ of this type form a subgroup of $\widetilde{\Aut}_0 X$, and this subgroup is reductive since it is isomorphic to $\prod_{\alpha\in\mathcal{C}}\mathrm{GL}(|\Delta_\alpha|,\mbc)$. So, the unipotent radical of $\widetilde{\Aut}_0 X$ should in some sense represent the influence of $B_\alpha$, which inspires us to define:
	\begin{definition}
            \item[$\bullet$] $G_s$ is the subgroup of $\widetilde{\Aut}_0 X$ consisting of elements with diagonal blocks of the form
		\[\begin{pmatrix}
			A_\alpha & 0 \\ 0 & C_\alpha
		\end{pmatrix}.\]
        
		\item[$\bullet$] $R_u$ is the subgroup of $\widetilde{\Aut}_0 X$ consisting of elements with diagonal blocks of the form
		\[\begin{pmatrix}
			I & B_\alpha \\ 0 & C_\alpha
		\end{pmatrix}.\]
	\end{definition}
	
        We claim:
	\begin{lemma}
		$R_s$ is the unipotent radical (Difinition \ref{reductivegp}) of $\widetilde{\Aut}_0 X$ and there is a split exact sequence
		\[\xymatrix{1 \ar[r] & R_u \ar[r] & \widetilde{\Aut}_0 X \ar[r] & \prod_{\alpha\in\mathcal{C}}\mathrm{GL}(|\Delta_\alpha|,\mbc) \ar[r] & 1}.\]
	\end{lemma}
	
	\begin{proof}
		First, for diagonal block $\Phi_\alpha$ of any elements $\Phi_\phi$ in $R_u$, $C_\alpha$ are upper triangle matrices with all eigenvalues $1$. This is deduced by the definition of the representation (Definition \ref{representation}) and noting that the components of $\phi$ look like 
        \[\phi_v=x_v+\hbox{higer-order terms}\] 
        and (see \eqref{phiDEq})
		\[\phi^D=\prod \phi_v^{a_v}=x^D+\hbox{higer-order terms}.\]  
		Thus all the elements in $R_u$ are unipotent. 

		Next, from the discussion above, we can identify $\prod_{\alpha\in\mathcal{C}}\mathrm{GL}(|\Delta_\alpha|,\mbc)$ with the subgroup $G_s$ of $\widetilde{\Aut}_0 X$ and describe the map from $\widetilde{\Aut}_0 X$ to $\prod_{\alpha\in\mathcal{C}}\mathrm{GL}(|\Delta_\alpha|,\mbc)$ by keeping $A_\alpha$, erasing $B_\alpha$ and recalculating $C_\alpha$ by the new $A_\alpha$ and $B_\alpha$. By definition, this is a surjective group homomorphism, and the kernel is clearly $R_u$. Our claim follows from this.
	\end{proof}
        
        As an immediate corollary, $\widetilde{\Aut}_0 X$ is connected, since both $R_u$ and $G_s$ are connected.
    
	By now, we have shown that:
	\begin{proposition}\label{Aut0structure}
		$\widetilde{\Aut}_0 X$ is a connected linear algebraic group of dimension $\sum|\Delta_\alpha|\dim S_\alpha$ \eqref{dimension}. Moreover, $\widetilde{\Aut}_0 X$ is the semi-direct product of its unipotent radical $R_u$ with a reductive group $G_s\cong\prod_{\alpha\in\mathcal{C}}\mathrm{GL}(|\Delta_\alpha|,\mbc)$. $(\mbc^*)^{\Delta_1}$ is a maximal torus of $G_s$, thus a maxiaml torus of $\widetilde{\Aut}_0 X$.
	\end{proposition}

    \begin{remark}\label{ruled surface2}
        Now we go back to Example \ref{ruled surface}, namely $\PP^2$ blown-up a point. Here $\mathcal{C}$ \eqref{LinearClassEq} contains three classes, $(1,0),(0,1)$ and $(1,1)$. The linear spaces $S_{(0,1)}$ and $S_{(1,0)}$ are spanned by monomials $y$ and $\{x'_1,x'_2\}$ respectively. The linear space $S_{(1,1)}$ is spanned by monomials $\{x_0,yx'_1,yx'_2\}$. So, the matrices in $\widetilde{\Aut}_0 X$ look like
    \[\begin{pmatrix}
	\Phi_{(0,1)} & 0 & 0 \\ 0 & \Phi_{(1,0)} & 0 \\ 0 & 0 & \Phi_{(1,1)}
    \end{pmatrix},\]
    where $\Phi_{(0,1)}$ is a nonzero complex number $a_y$, and $\Phi_{(1,0)}$ is a matrix $\begin{pmatrix}
	a_{1,1} & a_{1,2} \\ a_{2,1} & a_{2,2}
    \end{pmatrix}$ in $\mathrm{GL}(2,\mbc)$. $\Phi_{(1,1)}$      is of the form
    \[\begin{pmatrix}
	a_{0,0} & b_{0,1} & b_{0,2} \\ 0 & a_ya_{1,1} & a_ya_{1,2} \\ 0 & a_ya_{2,1} & a_ya_{2,2}
    \end{pmatrix}.\]
    
    We can define an isomorphism from $\Aut_0X$ to a subgroup $G'$ of $\mbc^*\times\mathrm{GL}(3,\mbc)$ by sending such a matrix to
    \[a_y\times\begin{pmatrix}
	a_{0,0} & b_{0,1} & b_{0,2} \\ 0 & a_ya_{1,1} & a_ya_{1,2} \\ 0 & a_ya_{2,1} & a_ya_{2,2}
    \end{pmatrix}.\]
    An element $(g_1,g_2)$ in $G=(\mbc^*)^2$ corresponds to $\begin{pmatrix}
	g_2 I_1 & 0 & 0 \\ 0 & g_1 I_2 & 0 \\ 0 & 0 & g_1g_2 I_3
    \end{pmatrix}$ in $\widetilde{\Aut}_0 X$ (here $I_i$ means the identity matrix of order $i$) and $g_2\times (g_1g_2 I_3)$ in $\mbc^*\times\mathrm{GL}(3,\mbc)$. In this way, the quotient $G'/(\mbc^*)^2$ gives the automorphism group that we have provided in Example \ref{ruled surface}.
    \end{remark}

\subsection{Step \texorpdfstring{$3.$}{3.} The relation between \texorpdfstring{$R_u$}{Ru} and \texorpdfstring{$R(P)$}{R(p)}}
    
	We can say more about the relation between $\widetilde{\Aut}_0 X$, $R_u$ and the polytope $P$ of $X$, or more precisely, the roots $R(P)$ (see Definition \ref{roots}).
	
	According to the representation constructed above (Definition \ref{representation}) and Proposition \ref{Aut0structure}, whether $\widetilde{\Aut}_0 X$ is reductive is equivalent to whether the orders of all $C_\alpha$ in the block $\Phi_\alpha$ (see \eqref{matrixform}) are all $0$, and further equivalent to whether there exists a non-prime effective divisor $D$ that is linear equivalent to some prime divisor $D_v$.
	
	Now suppose $D=\sum a_v D_v$ is such a divisor, then we can find a character $m\in M$ such that $(m)=D-D_{v_0}$ for some $v_0\in\Delta_1$. This means that for $v\in\Delta_1$,
    \begin{equation*}
        \langle m,v_0\rangle=-1,\ \langle m,v\rangle\geqslant 0\ for\ v\neq v_0.
    \end{equation*}
	The first equation shows that $m$ is in the codimension-$1$ face, denoted by $P_{v_0}$, of $P$ perpendicular to $v_0$, and the second equation shows that $m$ is in the relative interior of this face, therefore $m\in R(P)$.
	
	Conversely, for $m\in R(P)\cap P_{v_0}$, $(m)=D-D_{v_0}$ gives an effective divisor $D$ linear equivalent to $D_{v_0}$. Furthermore, $D$ is another prime divisor $D_{v'_0}$ which is equivalent to $-m\in R(P)\cap P_{v'_0}$ since $(-m)=D_{v_0}-D$.
	
	Therefore, the discussion above tells us: 
	
	\begin{lemma}\label{root-correspondence}
		There are one to one correspondences:
		\[R(P)\leftrightarrow\{(x_v, x^D)\,:\,\deg(x^D)=\deg(x_v),\ x_v\neq x^D\} \]
		\[R_s(P)\leftrightarrow\{(x_v, x_{v'})\,:\,\deg(v')=\deg(x_v),\ x_v\neq x_{v'}\}. \]
        Here the degree means the $\clx$-grading and $x^D$ is as in \eqref{xd}.
	\end{lemma}


	
	Just as the name "root" hint, similar to usual cases, the roots and the maximal torus can generate the whole group in the following sense: For every $m\in R(P)$, suppose $(m)=D-D_{v_0}$, we associate it with an one-parameter subgroup $\phi^{m,\lambda}\sse\widetilde{\Aut}_0 X$ named root automorphism by setting \beq\label{rootsAut}\phi^{m,\lambda}_{v_0}:=x_{v_0}+\lambda x^D,\;\;\; \phi^{m,\lambda}_v:=x_v \hbox{ for } v\neq v_0.\eeq
    When $m\in R_s(P)$, $\phi^{m,\lambda}$ lies in $G_s$ (see Proposition \ref{Aut0structure}), otherwise $\phi^{m,\lambda}$ lies in $R_u$. Then we claim:
	
	\begin{proposition}\label{Aut0Generator}
		$\widetilde{\Aut}_0 X$ is generated by the maxiaml torus $(\mbc^\ast)^{\Delta(1)}$ and the one-parameter subgroups $\phi^{m,\lambda}$ for $m\in R(P)$.
	\end{proposition}
	
	\begin{proof}
		For $m\in R_s(P)$, $\phi^{m,\lambda}$ correspond to the usual roots of the reductive group 
        \[G_s\cong\prod_{\alpha\in\mathcal{C}}\mathrm{GL}(|\Delta_\alpha|,\mbc),\]
        (namely, the matrices with $1$ in all diagonal positions, $\lambda\in\mbc^*$ in another position and $0$ in other positions). They together with the maximal torus $(\mbc^\ast)^{\Delta_1}$ generate $G_s$, since the tangent vectors of these one-parameter groups and the maximal torus in the identity of $\widetilde{\Aut}_0 X$ span the Lie algebra of it.
		
		For $m\in R_u(P)$, taking the derivative of $\Phi^{m,\lambda}$ (the image of $\phi^{m,\lambda}$ under the representation \eqref{representationEq}) with respect to $\lambda$, we get vectors 
        \[\theta^m:=\left.\frac{d}{d\lambda}\right|_{\lambda=0}\Phi^{m,\lambda}\] 
        in the Lie algebra of $R_u$ whose diagonal blocks are of the form
		\[\theta_\alpha^m=\begin{pmatrix}
			0 & E^m \\ 0 & C'_\alpha
		\end{pmatrix},\] 
		where $C'_\alpha$ is the derivative in corresponding place, $E^m$ has $1$ in the $(v_0,D)$ position and $0$ in other positions if $(m)=D-D_{v_0}$. Since $x^D$ does not contain $x^{v_0}$, $\Phi^{m,\lambda}$ has $1$ in the $(D,D)$-position and $0$ in other positions in the $D$-row, which shows that the $D$-row of $C'_\alpha$ is $0$ and for $k>1$
		\[(\theta_\alpha^m)^k=\begin{pmatrix}
			0 & 0 \\ 0 & D_\alpha^k
		\end{pmatrix}. \]
		Thus 
		\[\exp(\lambda\theta_m)=I+\sum_{k>0} (\theta^m)^k/k!=\Phi^{m,\lambda}.\]
		
		The vectors $\theta_m$ form a basis of the Lie algebra of $R_u$ because they are linearly independent and the number of them equals to the dimension of $R_u$. Therefore, a small open neighborhood of identity in $R_u$, thus $R_u$ itself, is generated by $\Phi^{m,\lambda}$ for $m\in R_u(P)$. 
		
		Combining these with the fact $\widetilde{\Aut}_0 X=G_s\ltimes R_u$ (Proposition \ref{Aut0structure}), we get what we want.
	\end{proof}

        \begin{remark}
        Example \ref{ruled surface} clearly shows the phenomenon discussed in this subsection:
        
        There are four roots in $R(P)$, $\{(0,1),(1,0),(-1,1),(1,-1)\}$. The latter two roots are semi-simple. Using the notation there, the roots $(0,1),(1,0)$ correspond to the divisors $E+D'_i-D_0$ and the pairs $(x_0, yx'_i)$ respectively. The one-parameter groups associated with them are exactly the first two listed in Example \ref{ruled surface}. The latter two roots $(-1,1),(1,-1)$ correspond to the divisors $D'_1-D'_2, D'_2-D'_1$ and the pairs $(x'_1, x'_2),(x'_2,x'_1)$ respectively. They are associated to the latter two one-parameter groups listed in Example \ref{ruled surface}.            
        \end{remark}

    \subsection{Step \texorpdfstring{$4.$}{4.} The key point in the proof of Proposition \ref{KeyExact}}\label{proofofkey}
    
	Finally, we come to the part skipped before: understanding the relation between $\widetilde{\Aut} X$, $\widetilde{\Aut}_0 X$ and $\Aut X$. The key point is an observation: 
	\begin{proposition}\label{AutClx}
		For any $\phi\in \aut\tx$, $\phi$ is in $\widetilde{\Aut}_0 X$ if and only if $it$ induces an automorphism on $X$, still denoted by $\phi$, which preserves all divisor classes. Formally, there is an exact sequence
		\beq\label{AutClxEq}\xymatrix{1 \ar[r] & G \ar[r] & \widetilde{\Aut}_0 X \ar[r] & \Aut X \ar[r] & \Aut\clx}.\eeq
        Here $\clx$ is the divisor class group of $X$ (Definition \ref{clx}). 
	\end{proposition}
	
	\begin{proof}
		One side of this is clear. $\widetilde{\Aut}_0 X$ is connected and $\Aut\clx$ is discrete, so the image of $\widetilde{\Aut}_0 X$ in $\Aut\clx$ must be the identity. Now we show that every automorphism $\phi$ on $X$ that preserves all divisor classes can be lifted to $\widetilde{\Aut}_0 X$.
		
		We denote by $\phi^*$ the pull-back map of $\phi$. Since $\phi^*(D_v)$ has the same degree as $D_v$, the divisor $\phi^*(D_v)-D_v$ is a principal divisor defined by a rational function $u_v$ on $X$. Pulling back $u_v$ to $\tx$ (see Theorem \ref{quotient-construction}) and multiplying by $x_v$, we get a regular function on $\tx$, thus on $\mbc^{\Delta_1}$. Then we can define
        \[\widetilde{\phi}_v:=u_v x_v\]
        and because $u_v$ is $G$-invariant, $\widetilde{\phi}_v$ is a polynomial of the same degree as $x_v$.
		
		We now combine $\widetilde{\phi}_v$ to obtain an automorphism $\widetilde{\phi}$. By definition we have $\widetilde{\phi}\in \widetilde{\Aut}_0 X$, so $\widetilde{\phi}$ descends to an automorphism, also denoted by $\widetilde{\phi}$, on $X$. However, it does not always coincide with the automorphism $\phi$. Luckily, we still have some room in $\widetilde{\phi}_v$ as we can multiply it by a non-zero constant $c_v$. 
        
        We first note that 
        \beq\label{adjust}
        \frac{m}{\widetilde{\phi}^*m}=\prod \frac{x_v^{\langle m,v \rangle}}{\phi^*x_v^{\langle m,v \rangle}}=\prod\frac{1}{u_v}
        \eeq
        and
        \[(\frac{m}{\phi^*m})=\sum D_v^{\langle m,v \rangle}-\sum\phi^*D_v^{\langle m,v \rangle}=(\prod\frac{1}{u_v}),\]
        so $\phi^*m$ and $\widetilde{\phi}^*m$ define the same divisor. Hence $\widetilde{\phi}^*m/\phi^*m$ is a constant. Therefore, $\widetilde{\phi}$ gives an element $t\in T'_X\cong\Hom_\ZZ(M,\mbc)$ by sending $m\in M$ to $\widetilde{\phi}^*m/\phi^*m$. 
        
        For an element $c=(c_1,...,c_v,...)\in (\mbc^*)^{\Delta(1)}$, we can adjust $\widetilde{\phi}$ by defining 
        \[c\widetilde{\phi}:=(c_v\widetilde{\phi}_v).\]
        Recalling the exact sequence \eqref{G-ExactSeq}, we denote the image of $c$ in $T'_X$ by $t_c$, then \eqref{adjust} tells us that $c\widetilde{\phi}$ corresponds to an element $t_c\cdot t\in\Hom_\ZZ(M,\mbc)$, where $\cdot$ means the multiplication in the group $T'_X$. By \eqref{G-ExactSeq}, we can find a $c'$ such that $t_{c'}=t^{-1}$.



		Now we denote $c'\widetilde{\phi}$ by $\widetilde{\phi'}$, and $\widetilde{\phi'}$ certainly descends to $X$. We can see $\widetilde{\phi'}^*m=\phi^*m$ for all characters $m \in M$, thus $\widetilde{\phi'}$ is lifted from $\phi$.
    \end{proof}

    Using this exact sequence and the connectedness argument, we can prove that
    \begin{corollary}\label{connectedness}
        The image of $\widetilde{\Aut}_0 X$ is the connected component $\Aut_0 X$ of $\Aut X$ containing the identity.
    \end{corollary}

    \begin{proof}
        Considering the exact sequence \eqref{AutClxEq} we just proved. $\Aut_0 X$ is connected, so it must be mapped to the identity of $\Aut \clx$, thus it is contained in the image of $\widetilde{\Aut}_0 X$. On the other hand, $\widetilde{\Aut}_0 X$ must be sent to $\Aut_0 X$ since it is connected.
    \end{proof}
    
    Therefore, $\Aut X/\Aut_0X$ is naturally embedded in $\Aut\clx$.

\subsection{Step \texorpdfstring{$5.$}{5.} Showing the role of \texorpdfstring{$\Aut \Delta$}{Aut Delta}
}

	Clearly, the next task is to determine the image of $\Aut X$ in $\Aut\clx$. Example \ref{square} brings about some natural automorphisms to our mind, the automorphisms induced by the automorphisms of the lattice $N$ that preserve the fan $\Delta$. 
    \begin{definition}\label{PDelta}
        The automorphism $\mathcal{P}\in\Aut \Delta$ \eqref{AutDeltaEq} permutes the vectors in $\Delta_1$, thus defining an automorphism $\phi_\mathcal{P}$ of $\mbc^{\Delta_1}$ by permuting the coordinates. All such automorphisms in $\Aut \mbc^{\Delta_1}$ form a group. We denote it by $\mathcal{P}(\Delta)$.
    \end{definition}
     By definition, $\mathcal{P}\mapsto \phi_\mathcal{P}$ is an isomorphism from $\Aut \Delta$ to $\mathcal{P}(\Delta)$.
	
	Conjugating an element $g\in G\sse(\mbc^\ast)^{\Delta_1}$ by $\phi_\mathcal{P}\in \mathcal{P}(\Delta)$ works as applying the corresponding permutation to the diagonal components of $g$. The result is clearly an element of $(\mbc^\ast)^{\Delta_1}$, and is actually contained in $G$ since, if $D,D'$ are linear equivalent by a character $m\in M$, then their images under the permutation are linear equivalent by $\mathcal{P}^*(m)$, where $\mathcal{P}^*$ means the dual action of $\mathcal{P}$ over $M$. In other words, there is a commutative diagram:
	\[\xymatrix{0 \ar[r] & M \ar[d] \ar[r] & \bigoplus_v\mbz D_v \ar[d] \ar[r] & \clx \ar[d] \ar[r] & 0  \\
		0 \ar[r] & M \ar[r] & \bigoplus_v\mbz D_v \ar[r] & \clx \ar[r] & 0}. \]
	The left vertical arrow is the dual action of $\mathcal{P}$ on $M$, the middle vertical arrow is the permutation corresponding to $\phi_\mathcal{P}$, and the right vertical arrow is induced by the five lemma. Applying the functor $\mathrm{Hom}_\mbz(\cdot, \mbc^*)$, we get an isomorphism of $G$, which is exactly the conjugation of $\phi_\mathcal{P}$. Therefore, we have $\mathcal{P}(\Delta)\sse \widetilde{\Aut} X$ and every $\phi_\mathcal{P}\in \mathcal{P}(\Delta)$ descends to an automorphism on $X$. We also use $\phi_\mathcal{P}$ to denote these automorphisms. Since $\mathcal{P}(\Delta)\cap G=\id$, we can view $\mathcal{P}(\Delta)$ as a subgroup of $\Aut X$.
	\\
    
	We now show that
    \begin{lemma}\label{PermuteClx}
        In the exact sequence \eqref{AutClxEq}, the image of $\mathcal{P}(\Delta)$ (Definition \ref{PDelta}) in $\Aut\clx$ coincides with the image of $\Aut X$.
    \end{lemma}

    \begin{proof}
        By Proposition \ref{Aut0structure} and the exact sequence \eqref{G-ExactSeq}, the natural torus action $T'_X\cong(\mbc^*)^{\Delta_1}/G$ is a maximal torus in $\Aut_0 X$. For any $\phi\in\Aut X$, $\phi T'_X\phi^{-1}$ is another maximal torus of $\Aut_0 X$, so we can find a $\phi'\in\Aut_0X$ to send it back to $T'_X$ by Theorem \ref{Torustheo}. Namely, \beq\label{conjugateeq}\phi'\phi T'_X\phi^{-1}(\phi')^{-1}=T'_X.\eeq $\phi'\phi$ is in the same connected component as $\phi$.
        
        We claim that if we act $\phi'\phi$ on $X$, we have $\phi'\phi(T_X)=T_X$ in $X$ (here $T_X$ is the $n$-dimensional orbit of $T'_X$-action in $X$): The complex torus $\phi'\phi(T_X)$ is preserved by the $T'_X$-action on $X$ by \eqref{conjugateeq}. But $\phi'\phi(T_X)\cap T_X\neq\emptyset$, so \[T'_X(\phi'\phi(T_X)\cap T_X)=T_X,\] which shows $\phi'\phi(T_X)\sse T_X$. But both $T_X$ and $\phi'\phi(T_X)$ are $n$-dimensional $T'_X$-orbits, thus $\phi'\phi(T_X)=T_X$.
	
	Next, we find an element $t\in T'_X$ such that $t\phi'\phi(T_X)$ preserves the identity, so $t\phi'\phi$ restricted to $T_X$ is a group automorphism. Therefore, $t\phi'\phi$ induces a linear automorphism of the lattice $M=\Hom_{\ZZ}(T_X,\mbc^*)$, the duality of which gives an automorphism on the lattice $N$, thus an element $\phi_\mathcal{P}$ in $\mathcal{P}(\Delta)$. The automorphism induced by $\phi_\mathcal{P}$ is actually $t\phi'\phi$ on $X$ since they map all the characters in the same way. Because $\phi_\mathcal{P}=t\phi'\phi$ is in the same connected component of $\phi$, they have the same image in $\Aut\clx$.
    \end{proof}

        By now, we can finally complete the proof of Proposition \ref{KeyExact}.

        \begin{proof}[Proof of Proposition \ref{KeyExact}]\label{proofKey}
        The discussion above Proposition \ref{KeyExact} has shown the exactness in places $G$ and $\widetilde{\Aut} X$. We only need to show that all automorphisms of $X$ can be lifted to $\widetilde{\Aut} X$. 
        
        Proposition \ref{AutClx} tells us how to lift the automorphisms in the identity component $\Aut_0 X$ to $\widetilde{\Aut}_0X$. Combining with Lemma \ref{PermuteClx}, we have $\Aut X\cong \Aut_0 X\rtimes\mathcal{P}(\Delta)$, but $\widetilde{\Aut}_0X\rtimes\mathcal{P}(\Delta)\sse\widetilde{\Aut}X$, so the descending map from $\widetilde{\Aut}X$ to $\Aut X$ is surjective. This completes the last part of Proposition \ref{KeyExact}.
        \end{proof}

	Furthermore, the image of $\mathcal{P}(\Delta)$ in $\Aut\clx$ is $\mathcal{P}(\Delta)/(\mathcal{P}(\Delta)\cap \widetilde{\Aut}_0 X)$. We claim:

    \begin{proposition}\label{WeylGroup}
        $\mathcal{P}(\Delta)\cap \widetilde{\Aut}_0 X$ is isomorphic to the Weyl group $W$ of $G_s$ with respect to the maximal torus $(\mbc^*)^{\Delta_1}$ (see Definition \ref{Weylgp}). 
        
        In addition, $\Aut X/\Aut_0X\cong\mathcal{P}(\Delta)/W \cong\Aut P/\Aut_0 P$ (see Definition \ref{Aut0P}).
    \end{proposition}

    \begin{proof}
        The Weyl group of $G_s$ in this case can be identified with a subgroup of $G_s$ consisting of matrices whose blocks are permutation matrices (see Example \ref{GLWeyl}). A permutation in $\mathcal{P}(\Delta)\cap \widetilde{\Aut}_0 X$ corresponds to a permutation matrix that preserves the degree, so it is in the Weyl group $W$ of $G_s$. On the other hand, any element in $W$ clearly preserves the torus $T\sse X$. The same argument as in the proof of Lemma \ref{PermuteClx} shows that it induces a linear automorphism on the lattice $N$ that preserves the fan $\Delta$, thus corresponding to an element in $\mathcal{P}(\Delta)$. Since $W\sse G_s\sse \widetilde{\Aut}_0 X$, we get $\mathcal{P}(\Delta)\cap \widetilde{\Aut}_0 X=W$.

        Now we consider the case where $X$ is smooth and Fano. By definition, $\mathcal{P}(\Delta)$ is naturally isomorphic to $\Aut\Delta$, and further isomorphic to $\Aut P$ by Lemma \ref{DualIso}. We only need to show that $W$ corresponds to $\Aut_0 P$ under this isomorphism.

        Every element in $W$ gives a grading-preserving permutation of $\Delta_1$, thus a grading-preserving permutation of the set of codimension-$1$ facets of $P$. Therefore, for a codimension-$1$ facet $F$ that is not preserved by this permutation, there should be a character $m$ in $R(P)$ such that $(m)$ is the difference of the divisors corresponding to $F$ and its image $F'$. Equivalently, $m\in -F\cap F'$. This is exactly the definition of $\Aut_0 P$.
    \end{proof}


	Summarize Proposition \ref{KeyExact}, \ref{Aut0structure}, \ref{Aut0Generator}, \ref{AutClx}, \ref{WeylGroup} and Lemma \ref{PermuteClx}, Theorem \ref{Demazure's structure theorem} we assumed in the last section follows:

    \begin{proof}[Proof of Theorem \ref{Demazure's structure theorem}]
    Proposition \ref{Aut0structure}, \ref{Aut0Generator}, Corollary \ref{connectedness} and the exact sequence \eqref{G-ExactSeq} tell us part (i) of Theorem \ref{Demazure's structure theorem}. 
    
    Combining Proposition \ref{KeyExact}, \ref{Aut0structure}, \ref{Aut0Generator}, and Corollary \ref{connectedness} we get part (ii) of Theorem \ref{Demazure's structure theorem}.

    Since $W\cap G$ is the identity, the image of $W$ in $G_s/G$ is the Weyl group of $G_s/G$ with respect to the maximal torus $T'_X=(\mbc^*)^n$. Therefore, Lemma \ref{PermuteClx} and Proposition \ref{WeylGroup} are exactly part (iii) of Theorem \ref{Demazure's structure theorem}.
    \end{proof}